\documentclass[numbers=withendperiod]{scrartcl}
\usepackage[margin=1in]{geometry}
\usepackage{libertinus}
\usepackage[T1]{fontenc}
\usepackage[utf8]{inputenc}
\usepackage[dvipsnames,table]{xcolor}
\usepackage{multirow}
\usepackage{float}
\usepackage[color=Orchid]{todonotes}
\usepackage[backend=bibtex,style=alphabetic, maxbibnames=4, maxalphanames=4]{biblatex} 
\AtBeginBibliography{\footnotesize}

\usepackage{amsmath}
\usepackage{amsthm}
\usepackage{amssymb}
\usepackage{mathrsfs}
\usepackage{dsfont}
\usepackage{tikz-cd}
\usepackage{mathtools}
\usepackage{hyperref}

\newtheorem{thm}{Theorem}[section]

\newtheorem*{thm*}{Theorem}
\newtheorem*{theorem*}{Theorem}
\newtheorem*{prop*}{Proposition}
\newtheorem*{lem*}{Lemma}
\newtheorem*{question*}{Question}

\newtheorem{prop}[thm]{Proposition}

\newtheorem{cor}[thm]{Corollary}

\newtheorem{lemma}[thm]{Lemma}
\newtheorem{lem}[thm]{Lemma}

\theoremstyle{definition}
\newtheorem{definition}[thm]{Definition}
\newtheorem{defn}[thm]{Definition}

\theoremstyle{remark}

\newtheorem{rmk}[thm]{\textsc{Remark}}

\newtheorem{example}[thm]{\textsc{Example}}

\newtheorem{notation}[thm]{\textsc{Notation}}

\newcommand{\operator}[1]{\operatorname{#1}}
\newcommand{\rk}{\operatorname{rk}}
\newcommand{\ONE}{\mathds{1}}
\newcommand{\id}{\operator{id}}

\newcommand{\mbf}[1]{\mathbf{#1}}

\newcommand{\mc}[1]{\mathcal{#1}}
\newcommand{\mscr}[1]{\mathscr{#1}}

\newcommand{\tensor}{\otimes}

\newcommand{\dirSum}{\oplus}

\newcommand{\comp}{\circ}

\newcommand{\iso}{\cong}

\newcommand{\define}{\coloneqq}

\newcommand{\into}{\hookrightarrow}
\newcommand{\onto}{\twoheadrightarrow}
\newcommand{\xto}{\xrightarrow}

\newcommand{\lowxto}[1]{\mathrel{\raisebox{-2pt}{$\xrightarrow{#1}$}}}
\newcommand{\lowxot}[1]{\mathrel{\raisebox{-2pt}{$\xleftarrow{#1}$}}}

\newcommand{\eqto}{\lowxto{\simeq}}
\newcommand{\eqot}{\lowxot{\simeq}}

\newcommand{\IN}{\mathds{N}}
\newcommand{\IZ}{\mathds{Z}}

\newcommand{\IR}{\mathds{R}}
\newcommand{\IC}{\mathds{C}}

\newcommand{\categoryname}{\mathbf}
\newcommand{\cat}[1]{\mbf{#1}} 
\newcommand{\sh}[1]{\mc{#1}}   
\newcommand{\st}[1]{\sh{#1}}   

\newcommand{\twocat}{\mscr}

\newcommand{\opp}{\mathrm{op}}

\newcommand{\point}{\ast}

\newcommand{\Man}{\cat{Man}}

\newcommand{\SmSt}{\cat{SmSt}}

\newcommand{\Rep}{\cat{Rep}}
\renewcommand{\Vec}{\cat{Vec}}

\newcommand{\sVec}{\cat{sVec}}
\newcommand{\Semi}{\cat{Semi}}
\newcommand{\bSemi}{\overline{\cat{Semi}}}

\newcommand{\Cinf}{{C^\infty}}

\newcommand{\DZ}{\mscr{Z}}
\newcommand{\CP}{\mathscr{C}\!\mathscr{P}}
\newcommand{\B}{\mathscr{B}}
\renewcommand{\H}{\mathrm{H}}
\newcommand{\ev}{\mathrm{ev}}

\newcommand{\Hom}{\mathrm{Hom}}
\newcommand{\End}{\mathrm{End}}

\newcommand{\Aut}{\mathrm{Aut}}

\newcommand{\Lie}{\operatorname{Lie}} 
\newcommand{\DiffSt}{\categoryname{DiffSt}}

\newcommand{\Bibun}{\categoryname{Bibun}}
\newcommand{\LieGpd}{\categoryname{LieGpd}}
\newcommand{\LieGpdW}{\categoryname{LieGpd[W^{-1}]}}
\newcommand{\Cat}{\categoryname{Cat}}
\newcommand{\Gpd}{\categoryname{Gpd}}

\newcommand{\ii}{\mathbf{i}}
\newcommand{\mker}{\mathrm{ker}}
\newcommand{\im}{\mathrm{im}}
\newcommand{\coim}{\mathrm{coim}}
\newcommand{\coker}{\mathrm{coker}}
\newcommand{\Spin}{\mathrm{Spin}}

\newcommand{\SU}{\mathrm{SU}}
\newcommand{\SO}{\mathrm{SO}}

\newcommand{\shff}{\#} 

\renewcommand{\restriction}{\mathord{\upharpoonright}}

\usepackage{breakurl}

\title{The Drinfel'd centres of String 2-groups}
\author{Christoph Weis}
\date{}

\begin{document}

\maketitle
\begin{abstract}
    Let $G$ be a compact connected Lie group
    and $k \in H^4(\B G,\IZ)$ a cohomology class.
    The String 2-group $G_k$ is the central extension 
    of $G$ by the 2-group $[\ast/U(1)]$ classified by $k$. 
    It has a close relationship to the level $k$ extension
    of the loop group $LG$.
    We compute the Drinfel'd centre of $G_k$ as a smooth 2-group.
    When $G$ is semisimple, we prove that the Drinfel'd centre
    is equal to the invertible part of the 
    category of positive energy representations of $LG$ at level $k$ 
    (as long as we exclude factors of $E_8$ at level 2).
\end{abstract}

\vspace{1.5em} 

\tableofcontents

\vspace{1.5em}

\addsec{Introduction}
Let $G$ be a compact connected Lie group
and $k \in \H^4(\B G,\IZ)$ a cohomology class.
The String 2-group $G_k$ is the central extension 
of $G$ by the 2-group $[\ast/U(1)]$ classified by $k$. 
It has a close relationship to the level $k$ extension
of the loop group $LG$.
In this note, we compute its Drinfel'd centre 
$\DZ G_k$ in the context of smooth 2-groups.
The result of this computation 
is interesting: we find that $\DZ G_k $ recovers the invertible part of $\Rep^k LG$ 
when $G$ is semisimple
(as long as we exclude factors of $E_8$ at level 2).

Before stating the result in more detail, we introduce the model of the String
2-group and centre we use.

\subsection*{The String 2-group}
A \emph{2-group}~\cite{baez2004higher2Gps} is a group object in the bicategory $\Cat$ of categories. 
It is a monoidal groupoid $(\cat{C},\ONE,\tensor,\omega)$ all of whose objects 
are invertible:
for all $x \in \cat{C}$, there exists an object $x^{-1}$ such that
$x \tensor x^{-1} \iso x^{-1} \tensor x \iso \ONE$.
The set of isomorphism classes of $\cat{C}$ forms a group $(\pi_0 \cat{C}, \tensor)$,
and the endomorphisms of $\ONE$ form an abelian group 
$(\pi_1 \cat{C},\comp=\tensor)$.\footnote{
There are two compatible group structures on $\pi_1 \cat{C}$ given by 
tensor product and composition. By the Eckmann-Hilton argument, these 
products agree and are commutative.}
Every object $x \in \pi_0\cat{C}$ acts on $\pi_1 \cat{C}$ by conjugation:
$f \mapsto \id_x \tensor f \tensor \id_{x^{-1}}$. This assembles into 
an action $\rho:\pi_0 \cat{C} \to \Aut(\pi_1 \cat{C})$.
The associator of $\cat{C}$ satisfies the pentagon equation
\begin{equation*}
    \rho(g)(\omega(g',g'',g''')) \omega(g,g'g'',g''') \omega(g',g'',g''') = 
    \omega(gg',g'',g''') \omega(g,g',g''g''')
\end{equation*}
for all $g,g',g'',g''' \in \pi_0 \cat{C}$. This is the equation of a 
group 3-cocycle in $Z^3(\B\pi_0\cat{C},\pi_1\cat{C})$.
Cohomologous 3-cocyles give rise to equivalent monoidal categories,
and 2-groups are completely classified by the data
$\left(\pi_0\cat{C},\pi_1\cat{C},\rho,
[\omega]\in \H^3(\B\pi_0\cat{C},\pi_1\cat{C})\right)$
\cite{sinh1975gr,baez2004higher2Gps}. 
The data of a (1-)group $G$ induces a 2-group (also denoted $G$) 
with objects $G$, tensor product the multiplication on $G$, and only identity morphisms. 
An abelian group $A$ also defines a 2-group $[\ast/A]$ 
with a single object $\ONE$, whose
endomorphisms form the group $\pi_1[\ast/A] = A$.\footnote{The notation 
indicates that $[\ast/A]$ is the quotient stack associated to the trivial $A$-action
on the point $\ast$.} 
Any 2-group $\cat{C}$ is an extension of 2-groups of this type:
\begin{equation*}
    [\ast/\pi_1\cat{C}] \to \cat{C} \to \pi_0\cat{C}.
\end{equation*}
The data of this extension is encoded by the conjugation action $\rho$ and 
the cohomology class $[\omega]$ of the associator.
When $\rho$ is trivial, one speaks of a \emph{central extension}.
Central extensions of $G$ by $[\ast/A]$ are thus classified by 
$[\omega] \in \H^3(\B G, A)$, 
where $A$ carries the trivial $G$-module structure.

\emph{Smooth 2-groups} are group objects in the 
bicategory of smooth groupoids (see Section~\ref{sec:smCats}).
Just as for discrete 2-groups, there are smooth 2-groups $G$ and $[\ast/A]$ associated
to a Lie group $G$ and an abelian Lie group $A$.
Central extensions of $G$ by $[\ast/A]$ are also classified by 
$\H^3(\B G,A)$~\cite{schommer2011central}, though one has to work with 
Segal-Mitchison cohomology~\cite{segal1970cohomology,brylinski2000differentiable} 
to make this precise.
We recall Segal-Mitchison cohomology in Section~\ref{sec:2Grps}.
Let $G$ be a compact connected Lie group. 
The short exact sequence of coefficients
$\IZ \to \IR \xto{\exp} U(1)$ gives rise to an isomorphism 
$\H^3(\B G,U(1)) \simeq \H^4(\B G,\IZ)$.
By abuse of notation, we write $k \in \H^4(\B G,\IZ)$ to denote a 
$U(1)$-valued 3-cocycle on $G$ representing $k$. For simple $G$, 
$\H^4(\B G,\IZ) \iso \IZ$, and $k$ is often called a \emph{level}.

The \emph{String 2-group} $G_k$ is the central extension of $G$
by $[\ast/U(1)]$ classified by $k$.
A large body of work has been devoted to constructing and understanding
the String 2-groups, with particular interest in the case where 
$G$ is simple simply-connected and $k$ is a
generator of $\H^4(\B G, \IZ)$: The 2-group $G_1$ is a model for 
the universal 3-connected cover of the Lie group $G$.
Such a 3-connected cover admits no incarnation as a finite-dimensional
Lie group~\cite[Footnote 2]{schommer2011central}, but it has been 
constructed variously as a topological 
group~\cite{stolz1996conjecture,stolz2004elliptic}, an
infinite-dimensional Lie 
(2-)group~\cite{baez2004higher2Gps, baez2007loop, henriques2008integrating, nikolaus2013smooth}, a
diffeological 2-group~\cite{waldorf2012construction},
and a smooth $\infty$-group~\cite{fiorenza2012vcech, bunk2020principal}.
We use the model of $G_k$ as a \emph{finite-dimensional} smooth 2-group
given in~\cite{schommer2011central}.

Let $LG=\Cinf(S^1,G)$ denote the loop group of $G$.
Transgression~\cite{brylinski1994geometry,waldorf2016transgressive}
establishes a
correspondence between String 2-groups $\{G_k\}$ and central extensions $\{LG_k\}$ 
of the loop group.\footnote{
The procedure requires a choice of connection on $G_k$. 
The correspondence between $\{G_k\}$ and $\{LG_k\}$ is bijective when $G$ is
simply-connected.}
The category $\Rep^k{LG}$ of \emph{positive energy representations} of $LG$ at level $k$
is a linear braided monoidal category, defined when $k$ satisfies a 
positivity condition.
It is in fact a \emph{modular tensor category}, 
and defines a 3-dimensional Topological Quantum Field Theory via 
the Reshetikhin-Turaev construction~\cite{reshetikhin1991invariants}.
This is understood to be Chern-Simons theory with gauge group $G$ at 
level $k$~\cite{freed2009remarks}, see~\cite{henriques2017chern} for 
an argument in the simply-connected case.


\subsection*{The smooth centre}
The centre of a monoid $M$ is the set of elements $z \in M$ such that 
$zm = mz$ for all $m \in M$. This concept admits a categorification to 
monoidal categories. The equality $zm = mz$ is replaced by the 
data of an isomorphism satisfying coherence conditions.
The \emph{Drinfel'd centre} $\DZ \cat{C}$ of a monoidal category $\cat{C}$
is the monoidal category whose objects are pairs
$(X, \gamma)$, where $X \in \cat{C}$ and $\gamma:X \tensor - \to - \tensor X$
is a natural isomorphism satisfying the hexagon equation
(recalled in Section~\ref{sec:smDZ}).
Such an isomorphism is called a \emph{half-braiding} for $X$.
Just as the centre of a monoid is a commutative monoid, the Drinfel'd centre of
$\cat{C}$ is a braided monoidal category.
The centre $\DZ \cat{C}$ of a 2-group $\cat{C}$
is again a 2-group (Lemma~\ref{lem:DZof2gpis2gp}).
The braiding $\beta$ makes $\DZ \cat{C}$ 
a \emph{braided categorical group}~\cite{joyal1993braided}.
The self-braidings $\beta_{x,x} \in \End(x \tensor x) = 
\pi_1 \DZ \cat{C} $ of objects $x \in \cat{C}$ assemble into a 
quadratic form~\cite[Ch 8.4]{etingof2016tensor}
\begin{align*}
    q: \pi_0 \DZ \cat{C} &\to \pi_1 \DZ \cat{C}\\
    x &\mapsto \beta_{x,x}.
\end{align*}
The map $q$ encodes $\DZ \cat{C}$ up to braided monoidal
equivalence~\cite{eilenberg1954groups}. 


If $\cat{C}$ is a monoidal category with a smooth structure, the Drinfel'd 
centre of the underlying monoidal category
has a distinguished subcategory on objects with smooth half-braidings. 
In this note, we compute this smooth Drinfel'd centre 
for the String 2-groups $G_k$. 
The result is another instance of the close relationship between $G_k$ 
and the associated central extension of the loop group, $LG_k$. 
We prove that for compact connected semisimple $G$ and positive-definite
$k \in \H^4(\B G, \IZ)$, the Drinfel'd centre of $G_k$ is
\begin{equation*}
    \DZ G_k \simeq {(\Rep^k{LG})}^\times,
\end{equation*}
the 2-group of invertible objects and invertible morphisms in 
$\Rep^k LG $, as long as we exclude factors of $E_8$ at level $k=2$.\footnote{
$\Rep^{k=2} LE_8 $ contains a non-trivial 
invertible object, while $\DZ E_{8,2}$ does not.
The invertibility of the non-identity object in $\Rep^{k=2} LE_8$ may be viewed as an accident of
low level. See for example the computation in~\cite{fuchs1991simple}.}
The proof of the above statement is indirect: we compute the left
hand side explicitly, and show the resulting braided categorical group 
agrees with that on the right hand side.
\begin{question*}
    Can all of $\Rep^k LG$ be recovered as a generalised centre 
    of the corresponding String 2-group $G_k$?
\end{question*}

\subsection*{Statement of Results}

Let $G$ be a compact simple simply-connected Lie group, and 
$\mathfrak{g}$ its complexified Lie algebra. 
There is a unique smallest positive-definite form
$I: \mathfrak{g} \tensor \mathfrak{g} \to \IC$ which is $\mathrm{Ad}_G$-invariant
and satisfies $I(X, X) \in 2 \IZ$ for all coroots of $\mathfrak{g}$ 
(a coroot is an element $X \in \mathfrak{g}$ satisfying 
$e^{2 \pi \ii X}=1$).\footnote{
For $\SU(n)$, $I$ is the trace of the product
of matrices: $I(X,Y)=\operatorname{tr}{X Y}$.}
We use $\exp: \IC \to \IC^\times$ to denote the map 
$w \mapsto e^{2 \pi \ii w}$. Its restriction $\exp:\IR \to U(1)$ 
has kernel $\IZ \into \IR$.

\begin{thm*}
    The Drinfel'd centre of the String 2-group $G_k$ is the 
    braided categorical group with $\pi_0 \DZ G_k = Z(G)$, 
    $\pi_1 \DZ G_k = U(1)$,
    and braided monoidal structure encoded by the quadratic form
    \begin{align*}
        q: Z(G) &\to U(1) \\
        z &\mapsto \exp \tfrac{k}{2} I(\bar{z},\bar{z}).
    \end{align*}
    Here, $\bar{z} \in \mathfrak{g}$ denotes an arbitrary lift of $z \in Z(G)$
    to the Lie algebra $\mathfrak{g}$ of $G$.
\end{thm*}
The inner product $\tfrac{1}{2}I(\bar{z},\bar{z})$ is a real number. 
Its values on different lifts of $z \in Z(G)$ differ by integers. 
By composing with $\exp$, we get a well-defined map $Z(G) \to U(1)$.

After a choice of maximal torus of $G$, every element of the centre
lifts to an element of the coweight lattice (see Section~\ref{sec:stringGrps}).
Their norms under $I$ can be found e.g. in~\cite{bourbaki1994lie}.
We list the results for Lie groups with non-trivial centre in 
Table~\ref{tab:resultsTable}. The centreless Lie groups $E_8,F_4,G_2$ all
have trivial Drinfel'd centre. 

\begin{table}[H]
\centering
\begin{tabular}{l|c|c|c}
Type       & $G_\mathrm{k}$      & $Z(G)$                                           & 
$q(z)=\exp\left(\tfrac{k}{2}I(\bar{z},\bar{z})\right)$ \\ \hline              
\rowcolor{Gray!20}
$A_{n-1}$  & $\SU(n)_k$      & $\IZ/n\langle\omega_1\rangle$               & 
$q(\omega_1)=\exp\left(\tfrac{k \cdot (n-1)}{2n}\right)$\\              
$B_{n \geq 2}$    & $\mathrm{Spin}(2n+1)_k$ & $\IZ/2\langle\omega_1\rangle$               & 
$q(\omega_1)=\exp\left(\tfrac{k}{2}\right)={(-1)}^k$\\              
\rowcolor{Gray!20}
$C_{n \geq 2}$    & $\mathrm{Sp}(n)_k$      & $\IZ/2\langle\omega_n\rangle$               & 
$q(\omega_n)=\exp\left(\tfrac{k\cdot n}{2}\right)={(-1)}^{k\cdot n}$\\              
$D_{2n+1}$   & $\mathrm{Spin}(4n+2)_k$   & $\IZ/4\langle\omega_{2n+1}\rangle$      & 
$q(\omega_{2n+1})=\exp\left(\tfrac{k \cdot (2n+1)}{8}\right)$\\              
\rowcolor{Gray!20}
\multirow{2}{*}{\centering $D_{2n\geq 4}$} & \multirow{2}{*}{$\mathrm{Spin}(4n)_k$}    &
\multirow{2}{*}{$\IZ/2\langle\omega_{2n-1},\omega_{2n}\rangle$} & 
\vtop{ 
\hbox{\strut $q(\omega_{2n-1})=q(\omega_{2n})=\exp\left(\tfrac{k\cdot n}{4}\right)={\ii}^{k\cdot n}$}
\hbox{\strut $q(\omega_{2n-1}+\omega_{2n})=\exp\left(\tfrac{k}{2}\right)={(-1)}^k$}
}\\ 
$E_{6}$     & $E_{6,k}$ & $\IZ/3\langle\omega_1\rangle$ &
$q(\omega_1)=\exp\left(\tfrac{2k}{3}\right)$\\              
\rowcolor{Gray!20}
$E_{7}$     & $E_{7,k}$ & $\IZ/2\langle\omega_7\rangle$ &
$q(\omega_1)=\exp\left(\tfrac{k}{4}\right)={\ii}^k$\\              
\end{tabular}
\caption{The Drinfel'd centre of $G_k$ (displayed for groups with nontrivial centre).
We denote by $\omega_i$ the $i$th \emph{co}weight.  The conventions for the numbering
are taken from~\cite{bourbaki1994lie}.} 
\label{tab:resultsTable}
\end{table}

A crucial step in our proof of the above theorem is the calculation of the 
centre of String 2-groups for $G=T$ a torus. 
They are also called \emph{categorical tori}~\cite{ganter2018categorical}. 
Let $\mathfrak{t}=\Lie(T)$ be the Lie algebra of a torus $T$, 
$\Lambda = \Hom(T,U(1)) \subset \mathfrak{t}^\ast$ its character 
lattice and $\Pi=\Hom(U(1),T) \subset \mathfrak{t}$ 
its cocharacter lattice.
The group $\H^4(\B T,\IZ)$ is the group of symmetric bilinear forms 
$\langle\cdot,\cdot\rangle:\mathfrak{t} \tensor \mathfrak{t} \to \IR$,
such that $\langle \pi, \pi \rangle \in 2\IZ$ for all $\pi \in \Pi$.
To describe the String 2-group associated to such a bilinear form,
we pick a (not-necessarily-symmetric) bilinear form 
$J:\mathfrak{t} \tensor \mathfrak{t} \to \IR$ which restricts to an integral
form on $\Pi \tensor \Pi$ and satisfies $J(x,y)+J(y,x)=-\langle x, y \rangle$ 
for all $x,y \in \mathfrak{t}$. 

The computation of the smooth Drinfel'd centre of the associated categorical 
torus $T_J$ was sketched in~\cite{freed2010topological}.
We provide a proof of the following statement in Section~\ref{sec:catTori}.
\begin{prop*}
    The Drinfel'd centre of $T_J$ has Lie group of objects 
    \begin{equation*}
        \pi_0 \DZ T_J = (\Lambda \dirSum \mathfrak{t})/\Pi, 
    \end{equation*}
    where the inclusion $\Pi \into \Lambda \dirSum \mathfrak{t}$ is 
    $\pi \mapsto (-J(\cdot,\pi)-J(\pi,\cdot),\pi)$.
    Further, $\pi_1 \DZ T_J = U(1)$, and the braided monoidal
    structure is encoded by the quadratic form
    \begin{align*}
        q: (\Lambda \dirSum \mathfrak{t})/\Pi &\to U(1) \\
        [\lambda,x] &\mapsto \lambda(x) \exp(J(x,x)).
    \end{align*}
\end{prop*}


Every compact connected Lie group $G$ fits into a short exact sequence
$Z \into \tilde{G} \onto G$. Here 
$\tilde{G} = T \times \Pi_i G_i$ is a product of simple
simply-connected groups $G_i$ with a torus $T$,
and $Z \subset Z(\tilde{G})$ is a finite subgroup of its 
centre~\cite[Cor V.5.31]{mimura1991topology}.
A central extension of $G$ by $[\ast/U(1)]$ 
pulls back to a central extension of $\tilde{G}$ by $[\ast/U(1)]$, 
giving the String 2-group $\tilde{G}_k$.
It is classified by the pullback degree 4 cohomology class in $\H^4(\B \tilde{G},\IZ) \simeq 
\H^4(\B T, \IZ) \times \Pi_i \H^4(\B G_i,\IZ)$ and we write $k=\left(J,\{k_i\}\right)$.

\begin{thm*}
    The group $Z=\ker{(\tilde{G} \onto G)}$ admits a unique lift to 
    $\DZ \tilde{G}_{k}$, and the quadratic form $q$ is trivial on $Z$.
    The Drinfel'd centre of $G_k$ is the subquotient
    \begin{equation*}
        \pi_0 \DZ G_k = Z^\bot/Z,
    \end{equation*}
    where $Z^\bot$ denotes the subgroup of $\pi_0\DZ \tilde{G}_k$ 
    on elements satisfying $q(x+z)=q(x)$ for all $z \in Z$.
    The quadratic form on $\pi_0 \DZ \tilde{G}_k$ descends to a quadratic 
    form on $\pi_0 \DZ G_k$, and this quadratic form describes
    the braided monoidal structure of $\DZ G_k$.

    The Drinfel'd centre of the covering String 2-group $\tilde{G}_{k} 
    = \times_i G_{i,k_i} \times T_J$ is given by
    \begin{equation*}
        \pi_0\DZ\tilde{G}_{\left(J,\{k_i\}\right)}=
        (\Lambda \dirSum \mathfrak{t})/\Pi \times \Pi_i G_{k_i},
    \end{equation*}
    $\pi_1 \DZ\tilde{G}_k=U(1)$ and quadratic form the product 
    of the quadratic forms $q_{k_i}$ and $q_J$ computed
    previously.
\end{thm*}

In the simply-connected case, every element of the 
centre admits a half-braiding,
but in general the subset of the centre admitting
a half-braiding varies with $k$. 
We present the result for the case $G=\SO(4)$ here 
(the computation is sketched in more detail in Example~\ref{ex:DZofSO4}).
The relevant cohomology group $\H^4(\B \SO(4),\IZ)$ has
two generators: the first Pontryagin class $p_1$ and the Euler class $\chi$. 
For the cohomology class $k=a \cdot p_1 + b \cdot \chi$, 
the Drinfel'd centre is given by
\begin{equation*}
    \DZ \SO(4)_k = 
    \begin{cases}
        \Vec_{\IZ/2}^\times & 2 a + b \equiv 0 \mod 4 \\
        \sVec^\times & 2 a + b \equiv 2 \mod 4 \\
        \Vec^\times & \text{ else.}
    \end{cases}
\end{equation*}
Here, $\Vec_{\IZ/2}^\times$ denotes the trivially braided 2-group 
$[\ast/U(1)] \times \IZ/2$,
$\sVec^\times$ is \emph{super}-$\IZ/2$ with non-trivial self-braiding of $-1 \in \IZ/2$,
and $\Vec^\times=[\ast/U(1)]$ is the trivial Drinfel'd centre whose only object
is the monoidal unit. 
As the notation indicates, they are the maximal braided sub-2-groups of the
braided fusion categories $\Vec[\IZ/2]$, $\sVec$ and $\Vec$.

\subsection*{Structure}
Section~\ref{sec:smCats} is a quick introduction to 
groupoids with smooth structure.
We review smooth 2-groups in Section~\ref{sec:2Grps},
recalling in particular the model for the String 2-groups 
given in~\cite{schommer2011central}.
In Section~\ref{sec:smDZ}, we discuss centres of smooth 2-groups 
as a special case of the notion of centre defined in~\cite{street2004centre}.
We compute this centre for categorical tori
in Section~\ref{sec:catTori} 
and for the String 2-groups in Section~\ref{sec:stringGrps}.
We end by showing that $\DZ G_k = {(\Rep^k LG)}^\times$
for $G$ compact connected semisimple.

\section*{Acknowledgements}
First and foremost, I thank my supervisor Andr\'{e} Henriques,
whose support and guidance have been invaluable to me.
I have benefitted from many more conversations while writing this paper. 
I particularly thank Luciana Basualdo Bonatto, Thibault D\'{e}coppet, 
Chris Douglas, Nora Ganter, Kiran Luecke, Jan Steinebrunner and Konrad Waldorf
for useful discussions and feedback on earlier drafts of this document.

\section{Smooth groupoids}
\label{sec:smCats}
We use $\Man$ to denote the category of paracompact smooth manifolds and smooth maps.
We equip it with the structure of a site by declaring
covers to be surjective submersions $f:Y \onto M$,
and denote the $n$-fold fibre product of a cover along itself
by $Y^{[n]} \define Y \times_M Y \times_M \cdots Y$.

\subsection*{Lie groupoids}
A \emph{Lie groupoid} $G_\bullet$ is a groupoid object in $\Man$.
Its data are two manifolds $G_1,G_0$ with a pair of
surjective submersions $s,t:G_1 \rightrightarrows G_0$, and smooth maps
\begin{align*}
    \comp: G_1 \times_{G_0} G_1 \to G_1 &&
    \id_{-}: G_0 \to G_1 &&
    (-)^{-1}: G_1 \to G_1,
\end{align*}
implementing composition, identities and inverses, respectively.
A \emph{smooth functor} $G_\bullet \to H_\bullet$ is a pair of 
smooth maps $G_0 \to H_0$, $G_1 \to H_1$ that preserve identities and composition.
A \emph{smooth natural transformation} between two smooth functors is a smooth map 
$G_0 \to H_1$ that makes the usual naturality diagrams commute ---
see~\cite{moerdijk1997orbifolds} for an introduction.

\begin{example}
\label{ex:ManLieGpd}
    Every manifold $M$ defines a Lie groupoid $M \rightrightarrows M$,
    with only identity morphisms.
\end{example}

\begin{example}
    A Lie group $G$ has an associated Lie groupoid $[\ast/G]\define G \rightrightarrows \ast$.
\end{example}

A smooth functor $F_\bullet:G_\bullet \to H_\bullet$ is \emph{fully faithful} if 
\begin{equation*}
    \begin{tikzcd}
        G_1 \dar["{(s,t)}"] \rar["F_1"] & H_1 \dar["{(s,t)}"] \\
        G_0  \times G_0 \rar["F_0 \times F_0"] & H_0 \times H_0
    \end{tikzcd}
\end{equation*}
is a pullback square, and it is 
\emph{essentially surjective} if
$s \comp p_2: G_0 \times_{H_0} H_1 \to H_0$ is a surjective submersion (the 
pullback $G_0 \times_{H_0} H_1$ is formed along the maps $F_0:G_0 \to H_0$ and
$t:H_1 \to H_0$).
A fully faithfully essentially surjective functor $F_\bullet$ is a 
\emph{smooth equivalence}~\cite{moerdijk1997orbifolds}.
Given a submersion $f:Y \to G_0$ from a manifold to the space of objects of a 
Lie groupoid $G_\bullet$, the \emph{pullback groupoid}
$f^\ast G_\bullet$ is the groupoid $f^\ast G_1 \rightrightarrows Y$,
with manifold of morphisms given by the pullback
\begin{equation*}
\begin{tikzcd}
f^\ast G_1 \arrow[d] \arrow[r]  & G_1 \arrow[d, "{(s,t)}"]                        \\
Y \times Y \arrow[r, "{f \times f}"] & G_0 \times G_0 
\arrow[lu, "\lrcorner" very near end, phantom].
\end{tikzcd}
\end{equation*}
It is constructed such that the natural functor
$f^\ast G_\bullet \to G_\bullet$ is fully faithful.
When $f$ is also surjective (and hence a cover), 
the functor $f^\ast G_\bullet \to G_\bullet$ is a
smooth equivalence.

\begin{example}
\label{ex:CechGpd}
    The \emph{\v Cech groupoid} $f^\ast M$ associated to a cover $f:Y \onto M$
    is $Y^{[2]} \rightrightarrows Y$ with composition
    given by the map $Y^{[2]}\times_Y Y^{[2]} = Y^{[3]} \to Y^{[2]}$
    projecting out the middle factor.
    This Lie groupoid is equivalent to $M \rightrightarrows M$.
\end{example}

A crucial example for our computation is the Lie groupoid built from
a smooth \v Cech 2-cocyle (Example~\ref{ex:cocycleGpd}).
We briefly recall how to compute smooth
\emph{\v Cech cohomology} $\check \H^\ast(M,A)$ of a manifold $M$ with coefficients in
an abelian Lie group $A$.
The cochain complex computing \v Cech cohomology
with respect to a cover $Y \onto M$ is given by
\begin{equation*}
    C^p_Y(M,A) \define \Cinf(Y^{[p+1]},A) 
\end{equation*}
with differential $d_\mathrm{Cech}$ the alternating sum over the pullbacks
along the maps $\delta_i:Y^{[p]} \to Y^{[p-1]}$ that project out the $i$-th factor:
\begin{equation*}
    \delta_i^\ast: \Cinf(Y^{[p-1]},A) \to \Cinf(Y^{[p]},A).
\end{equation*}
Any two covers $Y,Z \onto M$ admit a common refinement $Y \times_M Z \onto M$.
The assignment of cohomology groups assembles into a (contravariant) functor from the category of
covers of $M$ into the category of graded abelian groups.
\v Cech cohomology of $M$ is the colimit over this diagram. A cover is \emph{good}
if $Y^{[p]}$ is homotopy equivalent to a discrete space for all $p$,
ie.\ if each fibre product is a disjoint union of open balls. 
The collection of good covers is a cofinal subset of the poset of covers of $M$.
It is well-known that \v Cech cohomology with respect to any good cover
computes sheaf cohomology for a paracompact manifold.
Hence, the colimit we used to define \v Cech cohomology recovers sheaf cohomology,
and can be computed using \emph{any} good cover.

\begin{example}
\label{ex:cocycleGpd}
    A cocycle representing $[\lambda] \in \check \H^2(M,A)$ is a map 
    $\lambda:Y^{[3]} \to A$ for some cover $Y \onto M$.
    The \emph{principal $[\ast/A]$-bundle classified by $\lambda$}
    is the Lie groupoid
    $E^\lambda: Y^{[2]} \times A \rightrightarrows Y$,
    with source and target given by the two projections 
    $Y^{[2]} \rightrightarrows Y$,
    and composition given by 
    \begin{align*}
        (Y^{[2]} \times A) \times_{Y} (Y^{[2]} \times A) &\to 
        (Y^{[2]} \times A) \\
        \left((y_0,y_1,a), (y_1,y_2,b)\right) &\mapsto 
        (y_0,y_2,a+b+\lambda(y_0,y_1,y_2)).
    \end{align*}
    The cocycle condition ensures that this is associative 
    (a condition checked on $Y^{[4]}$).
\end{example}
    
Lie groupoids, smooth functors and smooth natural transformations
form a 2-category which we denote $\LieGpd$.
This 2-category of groupoids internal to $\Man$ 
is not a good 2-category of
smooth groupoids, because smooth functors are too strict
to implement the principle of equivalence:
A smooth equivalence $f^\ast G_\bullet \to G_\bullet$ 
usually does not admit an inverse.
Any sufficiently generic cover $f:\coprod_i U_i \onto M$ provides a counterexample. 

The standard way to proceed is to invert these morphisms by force.
We follow~\cite{pronk1996etendues} in writing $\LieGpdW$ to denote
the localisation of $\LieGpd$ at smooth equivalences.
Every 1-morphism in $\LieGpdW$ can be represented by an 
\emph{anafunctor}~\cite{roberts2012internal}: a span 
\begin{equation*}
    G_\bullet \eqot f^\ast G_\bullet \to H_\bullet,
\end{equation*}
whose left leg is a smooth equivalence.
Morphisms between anafunctors are defined as natural transformations on a common refinement 
of the involved covers~\cite{roberts2012internal}. In particular,
if two anafunctors have the same left leg $G_\bullet \eqot f^\ast G_\bullet$, 
then all morphisms between them are represented by natural transformations between their right legs
$f^\ast G_\bullet \to H_\bullet$.
The introduction of anafunctors manifestly inverts smooth equivalences. 
They can be viewed as smooth functors defined on some cover of the source.

\begin{example}
    The category of anafunctors from $M$ to $[\ast/G]$ is equivalent to
    the category of $G$-principal bundles over $M$. 
    The subcategory of smooth functors $M \to [\ast/G]$
    is that of topologically trivial bundles.
\end{example}

\subsection*{Differentiable stacks}
Every Lie groupoid $G_\bullet$ defines a functor of bicategories
$\Hom(-,G_\bullet):\Man^\opp \to \Gpd$, sending a manifold
$M$ to the groupoid of smooth functors $(M \rightrightarrows M) \to G_\bullet$.
Given any functor $\st{F}:\Man^\opp \to \Gpd$, there is a natural way to evaluate it
on the \v Cech groupoid (Example~\ref{ex:CechGpd}) associated to a cover $f:Y \onto M$
(we suppress natural 2-cells in the diagram):
\begin{equation*}
    \st{F}(f^\ast M) \define 
    \operator{lim}\left(\st{F}(Y) \rightrightarrows \st{F}(Y^{[2]}) \mathrel{\substack{\textstyle\rightarrow\\[-0.6ex]
              \textstyle\rightarrow \\[-0.6ex]
              \textstyle\rightarrow}}
              \st{F}(Y^{[3]})\right).
\end{equation*} 
The limit on the right hand side is represented by the so-called \emph{groupoid of descent data}
(see e.g.~\cite{vistoli2004notes}) for $\st{F}$ with respect to $Y$.
For $\st{F}=\Hom(-,G_\bullet)$, this is the groupoid of smooth
functors $\st{F}(f^\ast M) = \Hom(f^\ast M, G_\bullet)$.
The fact that smooth functors are too strict is reflected in the fact that
the restriction functor
$f^\ast \define \Hom(f,G_\bullet): \Hom(M,G_\bullet) \to \Hom(f^\ast M,G_\bullet)$ 
is generally not an equivalence of categories.

\begin{defn}
    A \emph{stack} (of groupoids over the site of manifolds) 
    is a functor of bicategories $\st{F}:\Man^\opp \to \Gpd$
    satisfying \emph{descent}:
    It sends coproducts to products $\st{F}(\coprod_i M_i) \eqto \prod_i \st{F}(M_i)$
    and for every cover $f:Y \onto M$, the restriction functor 
    $\st{F}(f):\st{F}(M) \to \st{F}(f^\ast M)$
    is an equivalence of groupoids.
\end{defn}

Stacks assemble into a bicategory, the full sub-bicategory
$\SmSt \subset [\Man^\opp,\Gpd]$ on functors which satisfy descent.
\begin{thm}\cite{giraud1966cohomologie}
    The forgetful functor $\SmSt \to [\Man^\opp, \Gpd]$ has a 
    left adjoint, the \emph{stackification functor}
    $L: [\Man^\opp, \Gpd] \to \SmSt$.
\end{thm}

\begin{defn}
    The \emph{stack presented by $G_\bullet$} is the 
    stackification of the associated 2-presheaf $\Hom(-,G_\bullet)$.
\end{defn}
The assignment $G_\bullet \mapsto \Hom(-,G_\bullet)^\shff$ 
defines a functor $\LieGpd \to \SmSt$.

\begin{defn}
    The category $\DiffSt$ of \emph{differentiable stacks} is the full
    subcategory of $\SmSt$ on the essential image of 
    $\LieGpd \to \SmSt$.
\end{defn}
\begin{thm}\cite{pronk1996etendues}
\label{thm:pronksThm}
     The functor $\LieGpd \to \DiffSt$, sending a Lie groupoid $G_\bullet$ 
     to the stackification of $\Hom(-,G_\bullet)$ 
     induces an equivalence
     \begin{align*}
         \LieGpdW \eqto \DiffSt.
     \end{align*}
\end{thm}
This tells us that inverting weak equivalences has the same effect as 
stackification. It allows us to use both the explicit description
of $\LieGpdW$ and the nice formal properties of $\DiffSt \subset \SmSt$.
We denote the stack presented by a Lie groupoid $G_\bullet$ 
by the same symbol, $G_\bullet:\Man^\opp \to \Gpd$.
Its value $G_\bullet(M)$ on a manifold $M$ is 
the groupoid of morphisms $M \to G_\bullet$ in $\LieGpdW$,
which can be computed as the groupoid of anafunctors $M \to G_\bullet$.

Every manifold $M$ defines a functor $\ev_M:\DiffSt \to \Gpd$, 
given by evaluation on $M$.
In particular, $\ev_\ast$ sends each stack 
$\st{X}$ to its \emph{groupoid of points} $\st{X}(\ast)$.
We view the remaining data of a differentiable stack as equipping $\st{X}(\ast)$
with a smooth structure.

\section{Smooth 2-groups}
\label{sec:2Grps}

\subsection*{From discrete to smooth 2-groups}
A \emph{(discrete) 2-group} is a monoidal category $(\cat{C}, \tensor, \ONE)$ 
with the property that every morphism in $\cat{C}$ has an inverse, and every 
object $x \in \cat{C}$ has a \emph{weak inverse}: 
an object $x^{-1} \in \cat{C}$ satisfying $x^{-1} \tensor x \iso \ONE \iso 
x \tensor x^{-1}$.
Up to equivalence, a 2-group $\cat{C}$ is characterised by 4 invariants
(see e.g.~\cite{baez2004higher2Gps}):
\begin{enumerate}
    \vspace{-0.5em}
    \itemsep-0.2em
    \item $\pi_0 \cat{C}$, the group of isomorphism classes of objects
    \item $\pi_1 \cat{C}$, the abelian group of endomorphisms of $\ONE$
    \item $\rho:\pi_0 \cat{C} \to \Aut(\pi_1 \cat{C})$, the conjugation action 
    $x \mapsto \id_x \tensor - \tensor \id_{x^{-1}}$
    \item $[\omega] \in \H^3(\B \pi_0 \cat{C}, \pi_1 \cat{C})$, the cohomology class of the associator.
    \vspace{-0.5em}
\end{enumerate}
We can build the 2-group associated to this data explicitly:
The underlying groupoid is 
\begin{equation*}
    \pi_0 \cat{C} \ltimes_\rho \pi_1 \cat{C} \rightrightarrows \pi_0 \cat{C},
\end{equation*}
with both source and target morphisms given by projection to $\pi_0 \cat{C}$.

The zero section $\pi_0 \cat{C} \to \pi_0 \cat{C} \ltimes_\rho \pi_1 \cat{C}$ serves 
as the
assignment of the identity morphism to each object. This canonically identifies the
endomorphism group of each object $g \in \pi_0 \cat{C}$ as $\End(g) = \pi_1 \cat{C}$.
There are no morphisms between objects corresponding to different elements 
$g \neq g' \in \pi_0 \cat{C}$. (Note that in the model of \emph{smooth} 2-groups we
use, it will not be possible to realise the String 2-groups in such a way.)
The composition of endomorphisms is group multiplication in $\pi_1 \cat{C}$.
The tensor product on objects and morphisms is given by the group 
multiplication in $\pi_0 \cat{C}$ and $\pi_0 \cat{C} \ltimes_\rho \pi_1 \cat{C}$,
respectively. 
The unitor isomorphisms are trivial.
An associator for the tensor structure must pick out an endomorphism
$\omega(g,g',g'') \in \End(g g' g'') = \pi_1 \cat{C}$ 
for each triple $g,g',g'' \in \pi_0 \cat{C}$.
Any cocycle
$\omega:{\pi_0 \cat{C}}^{\times 3} \to \pi_1 \cat{C}$ representing the cohomology class 
$[\omega]$ will do. 

The notion of a \emph{smooth} 2-group is the categorification of that 
of a Lie group (a smooth monoid with a smooth map sending
each element to its inverse).
One may avoid writing down conditions on the map $g \mapsto g^{-1}$ by
encoding the existence of the inverse indirectly:
\begin{definition}
    A \emph{Lie group} is a unital monoid $(G,\mu)$ in the 1-category 
    of smooth manifolds, such that the map 
    $(p_1,\mu):G \times G \to G \times G$ is a diffeomorphism.
\end{definition}
The smooth inverse appearing in the traditional definition of Lie groups 
can be recovered as
\begin{equation*}
\begin{tikzcd}
    G \ar[r,"i_1"] & G \times G \ar[r,"{(p_1,\mu)}^{-1}"] 
    & G \times G  \rar["p_2"] & G.
\end{tikzcd}
\end{equation*}
This definition now categorifies verbatim.
\begin{defn}
\label{defn:groupObject}
    A \emph{group object} in a bicategory with finite products
    is a unital monoid $(G_\bullet,\tensor)$ such that
    \begin{equation*}
        (p_1, \tensor): G_\bullet \times G_\bullet \to G_\bullet \times G_\bullet
    \end{equation*}
    is an 
    equivalence.\footnote{
         By a monoid object, we mean the maximally weak notion,
         often called a pseudomonoid. 
    }
\end{defn}

\begin{example}
    A \emph{discrete 2-group} is a group object in the bicategory of groupoids.
    That is a monoidal category in which each 
    object and morphism is invertible.
\end{example}

\begin{defn}~\cite{schommer2011central}
    A \emph{smooth 2-group} is a group object in $\DiffSt$.\footnote{
        In~\cite{schommer2011central}, 
        smooth 2-groups are defined internal to a bicategory denoted
        $\Bibun$. $\LieGpd$ includes fully faithfully into $\Bibun$, 
        and this induces an equivalence 
        $\DiffSt \simeq \LieGpdW \simeq \Bibun$~\cite{pronk1996etendues}.
    }
\end{defn}

\begin{example}
    Every Lie group $G$ defines a smooth 2-group.
    It has underlying Lie groupoid
    $G \rightrightarrows G$, with tensor product 
    given by multiplication.
\end{example}

\begin{example}
    Let $A$ be a Lie group.
    We denote by $[\ast/A]$ the 
    Lie groupoid
    $A \rightrightarrows \ast$.
    The multiplication on $A$ equips 
    it with a smooth 2-group structure iff
    $A$ is abelian.
\end{example}

\begin{example}
    The 2-category of discrete 2-groups is the full
    subcategory of the 2-category of smooth 2-groups
    on those whose underlying Lie groupoid
    is discrete.
\end{example}

\begin{example}
\label{ex:strict2Gps}
    A \emph{strict} smooth 2-group is a category object 
    in the category of Lie groups.
    Up to equivalence, one may always bring the underlying Lie groupoid 
    into the form 
    $G \ltimes_\rho H \rightrightarrows G$.
    Here, $G,H$ are Lie groups equipped with 
    a $G$-action $\rho:G \to \Aut(H)$, and a target
    homomorphism $t_0:H \to G$. 
    (This data $(G,H,t,\rho)$ satisfies further conditions 
    and is known as a \emph{crossed module}.)
    The source map can be chosen to be projection to $G$,
    and the target map is given by
    $t:(g,h) \mapsto t_0(h)g$.
    Composition is multiplication in $H$
    (i.e.\ $(t_0(h)g,h') \comp (g,h)=(t_0(h'h)g,h'h)$),
    tensor product is multiplication in the 
    object and morphism groups,
    and all the other data is trivial.
    This example is spelt out in great detail 
    in~\cite{porst2008strict}.
\end{example}

\subsection*{String 2-groups and Segal-Mitchison cohomology}
String 2-groups are smooth 2-groups associated to a Lie group $G$ 
and a cohomology class.
The underlying discrete 2-group of a String 2-group 
is a 2-group with $\pi_1 = U(1)$ and trivial action $\rho:\pi_0 \to \Aut(\pi_1)$.
These discrete String 2-groups are classified by their underlying 
``object group'' $\pi_0 = G^\delta$ and 
the class of the associator $k \in \H^3(\B G^\delta,U(1))$ --- 
the latter is a class in discrete group cohomology.
To understand String 2-groups as \emph{smooth} 2-groups,
we must talk about Lie group cohomology.

\emph{Globally smooth group cohomology}~\cite{stasheff1978continuous,blanc1985homologie}
$\H^\bullet_{\mathrm{sm}}(\B G,A)$ of a Lie group $G$ with coefficients in a
smooth $G$-module $A$ is computed just as in the discrete case,
except cocycles $G^{\times n} \to A$ are required to be smooth maps.
This cohomology theory is \emph{not} very well-behaved. A short exact sequence
of coefficients does not give rise to a long exact sequence in 
globally smooth group cohomology. Further, the group $\H^2_{\mathrm{sm}}(\B G,A)$, 
which ordinarily classifies extensions
of $G$ by $A$, only detects those extensions $A \to E \to G$ where
$E$ is topologically a trivial $A$-bundle over $G$.
One may fix this problem by only requiring smoothness in a neighbourhood
of the identity of $G$. The resulting cohomology theory is called
\emph{locally smooth cohomology}.
More geometric is the cohomology theory introduced by Segal 
in~\cite{segal1970cohomology} and recalled below.
It is equivalent to locally smooth cohomology by~\cite{wagemann2015cocycle}.


Lie group cohomology of $G$ is the cohomology of its \emph{classifying space} $\B G$. 
We use the model of $\B G$ as a 
simplicial manifold: the manifold of $q$-simplices is $\B G_q = G^{q}$ and the
face maps $d_i:G^{q} \to G^{q-1}$ are given by
\begin{equation*}
    d_i: (g_1, \ldots, g_q) \mapsto 
    \begin{cases}
        (g_2, \ldots, g_q) & i=0 \\
        (\ldots, g_i g_{i+1} ,\ldots) & 0 < i < q \\
        (g_1, \ldots, g_{q-1}) & i=q.
    \end{cases}
\end{equation*}

This simplicial manifold contains more homotopical information than
the topological space $|\B G|$ obtained from $\B G$ 
by geometric realisation. The difference is detected by Segal-Mitchison cohomology
unless the coefficient group is discrete (see Example~\ref{ex:special_case_Segal_cohomology}).

A \emph{simplicial cover} $Y_\bullet \onto \B G$ 
is a collection of covers $Y_{q} \onto \B G_q$, together with simplicial maps
(which we will also denote by $d_i$)
such that the corresponding squares in the diagram below 
commute. 
\begin{equation*}
\begin{tikzcd}
Y_{0} \arrow[d, two heads] & Y_{1} \arrow[d, two heads] \arrow[l, shift right=0.5] \arrow[l, shift left=0.5] & Y_{2} \arrow[d, two heads] \arrow[l, shift left=1] \arrow[l, shift right=1] \arrow[l] & Y_{3} \arrow[l, shift left=1.5] \arrow[l, shift left=0.5] \arrow[l, shift right=0.5] \arrow[l, shift right=1.5] \arrow[d, two heads] & {} \arrow[l, shift right=2] \arrow[l, shift right=1] \arrow[l] \arrow[l, shift left=1] \arrow[l, shift left=2] & \cdots \\
\ast                      & G \arrow[l, shift right=0.5] \arrow[l, shift left=0.5]                            & G \times G \arrow[l] \arrow[l, shift right=1] \arrow[l, shift left=1]                   & G \times G \times G \arrow[l, shift left=1.5] \arrow[l, shift left=0.5] \arrow[l, shift right=0.5] \arrow[l, shift right=1.5]          & {} \arrow[l] \arrow[l, shift left=2] \arrow[l, shift left=1] \arrow[l, shift right=2] \arrow[l, shift right=1] & \cdots
\end{tikzcd}
\end{equation*}
Associated to a simplicial cover is a \v{C}ech-simplicial 
double complex~\cite{brylinski2000differentiable} (recall that $Y^{[p]}$ denotes the $p$-fold fibre product of $Y$ over $M$):
\begin{align*}
    C^{p,q}(\B G,A) = \Cinf(Y_{q}^{[p+1]},A) && 
    d = d_\mathrm{\check{C}ech} + d_\mathrm{simp},
\end{align*}
where $d_\mathrm{\check{C}ech}$ denotes the \v Cech differential
and $d_\mathrm{simp}$ is the alternating sum of the pullback maps
\begin{equation*}
    d_i^\ast: \Cinf(Y_{q}^{[p]},A) \to \Cinf(Y_{q+1}^{[p]},A).
\end{equation*}
A simplical cover is \emph{good} if $Y_i \onto BG_i$ is good for all $i$.
For good simplicial covers of $\B G$, we may always choose $Y_{0}$ to be a point.
\begin{definition}
\label{def:SegalMitchisonCohomology}
    \emph{Segal-Mitchison cohomology} of $G$ with coefficients in 
    an abelian Lie group $A$ is the cohomology of the \v Cech-simplicial complex
    associated to a good simplicial cover $Y_\bullet \onto \B G$.
\end{definition}
This is independent of the good cover chosen~\cite{brylinski2000differentiable}.
Segal-Mitchison cohomology can be defined more generally with coefficients in 
any smooth $G$-module, but we will only need the case with trivial action.
\begin{example}
\label{ex:special_case_Segal_cohomology}
    For certain coefficients $A$, Segal-Mitchison cohomology reduces to other 
    cohomology theories~\cite{segal1970cohomology,brylinski2000differentiable}:
    \begin{itemize}
        \vspace{-0.5em}
        \itemsep-0.2em
        \item If $A$ is a vector space, it agrees with globally smooth group cohomology.
        \item If $A$ is discrete, it is the (singular) cohomology of the 
        topological space $|\B G|$.
        \vspace{-0.5em}
    \end{itemize}
\end{example}
Segal-Mitchison cohomology associates to any short exact sequence of coefficients a
long exact sequence of cohomology groups. The exponential exact sequence 
$\IZ \into \IR \onto U(1)$ gives rise to a long exact sequence
\begin{equation*}
    \cdots \to \H^p(\B G,\IR) \to \H^p(\B G,U(1)) \xrightarrow{\delta} 
    \H^{p+1}(\B G,\IZ) \to \H^{p+1}(\B G,\IR) \to \cdots.
\end{equation*}
Note that $\H^\bullet(\B G, \IR)$ is cohomology with respect to the sheaf of smooth
$\IR$-valued functions, not the sheaf of locally constant $\IR$-valued functions. 
In particular, it is \emph{not} the same as singular cohomology 
of $|\B G|$ with coefficients in discrete $\IR$ and thus does not see the usual real
characteristic classes.
When $G$ is compact, $\H^\bullet(\B G,\IR)$ vanishes in positive 
degrees~\cite{blanc1985homologie} (see also~\cite[Cor 97]{schommer2011central}), so
the map $\delta$ is an isomorphism for $p>0$. In particular, 
$\H^3(\B G,U(1)) \simeq \H^4(\B G,\IZ)$.
If $G$ is also simply-connected and simple, then 
$G$ is in fact 2-connected and $\pi_3 G = \IZ$~\cite[Thm VI.4.17]{mimura1991topology}\cite{MOHomotopygroupsofLiegroups}.
The group $\H^4(\B G,\IZ)$ can be computed as the simplicial cohomology of 
the geometric realisation $|\B G|$ (see Example~\ref{ex:special_case_Segal_cohomology}). 
This topological space has homotopy groups
$\pi_{i<4} |\B G|=0, \pi_4 |\B G|=\IZ$.  
The Hurewicz isomorphism and Universal Coefficient Theorem imply
\begin{equation*}
    \H^3(\B G,U(1)) \simeq \H^4(\B G,\IZ) \iso \IZ.
\end{equation*}


From a Segal-Mitchison 3-cocycle, one may build a smooth 2-group~\cite{schommer2011central}. 
We recall this construction now.
Let $G$ be a Lie group, $A$ an abelian Lie group, and 
$k \in \H^3(\B G,A)$ a class in Segal-Mitchison cohomology.
A Segal-Mitchison cocycle representing a cohomology class in $\H^3(\B G, A)$ is a
simplicial cover $Y_\bullet \onto \B G$ and a triple\footnote{
The \v Cech-simplicial complex has four groups in cohomological degree 3,
but the cocycle condition implies that the component in 
$\Cinf(Y_{0}^{[4]}, A)$ is trivial.}
\begin{equation*}
    (\lambda, \mu, \omega) \in \Cinf(Y_{1}^{[3]},A) \times
    \Cinf(Y_{2}^{[2]}, A) \times
    \Cinf(Y_{3}, A).
\end{equation*}
Recall the Lie groupoid 
$E^\lambda: Y_{1}^{[2]} \times A \rightrightarrows Y_{1}$
associated to the \v Cech 2-cocycle $\lambda$ from Example~\ref{ex:cocycleGpd}.
Using the map $(d_0, d_2):Y_2 \to Y_1 \times Y_1$, we can form the pullback
groupoid $F^\lambda = {(d_0, d_2)}^\ast(E^\lambda \times E^\lambda)
= Y_2^{[2]} \times A \times A \rightrightarrows Y_2$.
The tensor structure on $E^\lambda$ is given by the anafunctor
\begin{align*}
    E^\lambda \times E^\lambda \eqot F^\lambda \to E^\lambda,
\end{align*}
where the map $F^\lambda \to E^\lambda$ is given by $d_1:Y_2 \to Y_1$ on 
objects, and by
\begin{align*}
    Y_2^{[2]} \times A \times A &\to Y_1^{[2]} \times A \\
                   (v_0,v_1,a,b) &\mapsto (d_1(v_0),d_1(v_1),a+b+\mu(v_0,v_1)),
\end{align*}
on morphisms.
One may further pull back along $(d_0 d_0, d_2 d_0, d_1 d_0): Y_3 \to Y_1^{\times 3}$ 
to obtain another Lie groupoid $H^\lambda$.
Then $(- \tensor -) \tensor -$ and $- \tensor (- \tensor -)$ are represented 
by spans $(E^\lambda)^{\times 3} \eqot H^\lambda \to E^\lambda$,
and $\omega \in \Cinf(Y_3,A)$ defines a smooth natural transformation 
between them.
The cocycle condition ensures that the above indeed defines a smooth 2-group.

There is a notion of \emph{central extension} of smooth 
2-groups~\cite[Defn 83]{schommer2011central} paralleling the 
definition for groups.
Equivalence classes of central extensions of $G$ by $[\ast/A]$ are 
in bijective correspondence with elements of $\H^3(\B G,A)$, and 
a representative of this equivalence class may be built from a cocycle as
above. 
The smooth String 2-group $G_k$ is the central extension of 
$G$ by $[\ast/U(1)]$ corresponding to
$k \in \H^3(\B G, U(1))$~\cite[Thm 100]{schommer2011central}.

\section{The smooth centre}
\label{sec:smDZ}
The Drinfel'd centre of a monoidal category $(\cat{C},\tensor,\ONE, \omega)$ 
is the category whose objects are pairs $(X, \gamma)$, where $X \in \cat{C}$ and 
$\gamma: X \tensor - \eqto - \tensor X$
is a \emph{half-braiding}:
a natural isomorphism satisfying the \emph{hexagon equation}
\begin{equation*}
    \omega(Y,Z,X) \comp \gamma(Y \tensor Z) \comp \omega(X,Y,Z) = 
    \left(\gamma(Y) \tensor \id_Z\right) \comp \omega(Y,X,Z) \comp
    \left(\id_Y \tensor \gamma(Z)\right).
\end{equation*}
It is the analogue of the centre of a monoid in the world of 
monoidal categories. 
We study the corresponding notion in the context of smooth 2-groups.

\subsection*{Drinfel'd centres of smooth 2-groups}
In~\cite{street2004centre}, the notion of the Drinfel'd centre
was internalised to any braided monoidal bicategory (and thus in particular 
to the symmetric monoidal bicategory of stacks). 
Let $\twocat{B}$ be a symmetric monoidal bicategory with product 
$\boxtimes$ and braiding $b$, and
$(\cat{C},\tensor)$ a monoid object in $\twocat{B}$. 
\begin{defn}
\label{def:centrePiece}
    Let $U \in \twocat{B}$.
    A $U$-family of \emph{centre pieces} for $(\cat{C}, \tensor)$ is a pair
    $(u , \gamma)$ where $u: U \to \cat{C}$ is a morphism and 
    \begin{equation*}
    \gamma: \tensor \comp (u \boxtimes \id_\cat{C}) \eqto 
    \tensor \comp (\id_\cat{C} \boxtimes\ u) \comp b_{U,\cat{C}}
    \end{equation*}
    an invertible 2-cell satisfying the hexagon equation
    (phrased internal to $\twocat{B}$).
    We think of $u$ as a $U$-point of $\cat{C}$ and call 
    $\gamma$ a \emph{half-braiding} for $u$.
    A \emph{morphism of centre pieces} 
    $(u, \gamma) \to (u^\prime, \gamma^\prime)$ is 
    a morphism $u \to u^\prime$ which commutes 
    with the half-braidings.
    We denote by $\CP(U,\cat{C})$ the category of $U$-families of 
    centre pieces for $(\cat{C},\tensor)$.
\end{defn}

Half-braidings may be pulled back along maps $U^\prime \to U$,
and $\CP(-,\cat{C})$ admits the structure of a 2-presheaf over
$\Man$.

\begin{defn}
    The centre $\DZ \cat{C}$ of a monoid 
    $\cat{C}$ in a braided monoidal bicategory $\twocat{B}$
    is the representing object for the 2-presheaf
    $\CP(-,\cat{C}):\twocat{B}^\opp \to \Cat$.
\end{defn}

In~\cite{pirashvili2021centre}, this centre was computed explicitly 
in the bicategory of crossed modules (see Example~\ref{ex:strict2Gps}). 
We now specialise to $\twocat{B} = \SmSt$, the bicategory of stacks.
The 2-Yoneda Lemma says that for any stack $\st{F}$ and
manifold $M$, the category $\st{F}(M)$ is naturally equivalent to the category of
1-morphisms $M \to \st{F}$~\cite{lerman2010orbifolds}.
Thus, a morphism $u:M \to A$
is equivalently an object of the category $A(M)$,
which makes Definition~\ref{def:centrePiece} 
a parameterised version of the ordinary 
Drinfel'd centre in this case.

\begin{thm}\cite{street2004centre}
    The centre $\DZ \cat{C}$ of a monoid $\cat{C} \in \twocat{B}$ exists
    if $\twocat{B}$ is finitely complete and closed. 
    Denote the internal $\Hom$ of $\twocat{B}$ by $[\cdot,\cdot]$,
    then the centre is the limit
    \begin{equation*}
        \DZ\cat{C}
        =
        \lim \big(
        \cat{C}
        \mathrel{\substack{\textstyle\rightarrow\\[-0.6ex]
                      \textstyle\rightarrow}} 
        [\cat{C}, \cat{C}]
        \mathrel{\substack{\textstyle\rightarrow\\[-0.6ex]
                      \textstyle\rightarrow \\[-0.6ex]
                      \textstyle\rightarrow}} 
        [\cat{C} \boxtimes \cat{C}, \cat{C}]
        \big).\footnote{
            The morphisms in the diagram are induced by the monoid
            structure of $\cat{C}$. We suppress the 2-cells.
        }
    \end{equation*}
    The centre is equipped with a monoidal product $\tensor$, 
    a monoidal morphism $\DZ\cat{C} \to \cat{C}$, and a braiding
    $\beta: \tensor \eqto \tensor \comp b_{\DZ \cat{C},\DZ \cat{C}}$.
\end{thm}

The bicategory $\SmSt$ is finitely complete~\cite{street1982characterization}, 
so in particular fulfills the assumptions of the above theorem. 
Any smooth 2-group $G \in \DiffSt \subset \SmSt$ thus has a centre 
$\DZ G \in \SmSt$.
The functor $\ev_U: \SmSt \to \Gpd$ which evaluates
a stack on a manifold $U$ induces a braided monoidal comparison functor
$(\DZ G)(U) \to \DZ(G(U))$. This functor is the inclusion of the smooth half-braidings
and smooth maps of the associated centre pieces.

The monoidal groupoid $G(U)$ is a 2-group, and hence every centre piece $(X, \gamma)$ 
has an inverse $(iX, i\gamma^{-1}_i)$~\cite[Ch 7.13]{etingof2016tensor}, where
$i:G(U) \to G(U)$ is a functor assigning inverses.

We would like to show that this inverse exists
as a smooth half-braiding in $(\DZ G)(U)$. 
One can pick a global inverse functor $i_{\mathrm{glob}}: G \to G$ and directly show this, but 
the diagrams involved are unwieldy due to the presence of coherences.
By uniqueness of inverses (up to isomorphism), it suffices to 
show the existence of inverses in $(\DZ G)(U)$. 
\begin{lemma}
\label{lem:DZof2gpis2gp}
    The Drinfel'd centre of a group object in $\SmSt$ is a group 
    object in $\SmSt$.
\end{lemma}
\begin{proof}
    This is completely formal.
    The inclusion of group objects in $\SmSt$ into monoid objects in $\SmSt$
    is a reflective localisation of 
    $(2,1)$-categories~\cite[Cor 59-Thm 61]{schommer2011central}, and thus
    of $(\infty,1)$-categories. 
    Reflective localisations of $(\infty,1)$-categories reflect limits~
    \cite{MOreflectiveLocalization}
    so the limit exists in the bicategory of 2-groups, and agrees with the
    limit calculated in the bicategory of monoid objects.
\end{proof}
\begin{cor}
\label{cor:centreOf2GpIsNice2Gp}
    The centre $\DZ G$ of any 2-group in $\DiffSt$ is a 2-group object in $\SmSt$.
    Given a global inverse $i: G \to G$, the inverse of a centre piece
    $(X, \gamma) \in (\DZ G)(U)$ is given by $(i_U X, i_U\gamma^{-1}_{i_U})$.
\end{cor}
\begin{proof}
    Lemma~\ref{lem:DZof2gpis2gp} guarantees the existence of an inverse to each centre piece 
    $(X,\gamma)$. Under the functor $\ev_U:\SmSt \to \Gpd$, this inverse must
    be isomorphic to $(iX, i\gamma^{-1}_i)$ for any choice of inverse map $i$.
\end{proof}

\begin{lemma}
\label{lem:pi1ZCispi1C}
    Let $\cat{C}$ be a 2-group. Then $\pi_1 \DZ \cat{C} = {(\pi_1 \cat{C})}^{\pi_0 \cat{C}}$,
    the invariants under the conjugation action $\rho$.
\end{lemma}
\begin{proof}
    The trivial centre piece (the monoidal unit of $\DZ \cat{C}$) 
    is the unit morphism $\ONE: \point \to \cat{C}$, with half-braiding 
    $\gamma_\ONE$ built as a composite of unitor morphisms
    $\ONE \tensor - \to - \to - \tensor \ONE$.
    The condition for a morphism $f \in \End(\ONE) = \pi_1 \cat{C}$ to be an endomorphism 
    of $(\ONE,\gamma_\ONE)$ is
    \begin{equation*}
        \gamma_\ONE(x) \comp (f \tensor \id_x) = 
        (\id_x \tensor f) \comp \gamma_\ONE(x)
    \end{equation*}
    for all $x \in \cat{C}$. Tensoring with $\id_{x^{-1}}$ shows that this is equivalent
    to $f$ being conjugation-invariant.
\end{proof}

The central extensions we are interested in have trivial
conjugation action, and so $\pi_1 \DZ \cat{C} = \pi_1 \cat{C}$ 
in this case.

\subsection*{Braided 2-groups and quadratic forms}
Recall that a braiding $\beta$ on a monoid 
$(\cat{B},\tensor,\ldots)$ is a 2-morphism 
$\beta: - \tensor - \to - \tensor^\opp -$ 
which is compatible with units and satisfies the hexagon equation in 
both variables.
\begin{defn}
    A \emph{braided} smooth 2-group is a braided monoid
    $(\cat{B}, \tensor, \omega, l, r, \beta) \in \DiffSt$ 
    whose underlying monoid $(\cat{B}, \tensor, \omega, l, r)$ is a smooth 2-group.
\end{defn}

A \emph{discrete} braided 2-group is a braided monoidal 
groupoid whose underlying monoidal groupoid is a 2-group. 
They are also known as \emph{braided categorical groups}
\cite{joyal1993braided}.
The existence of a braiding on $\cat{B}$ forces
$\pi_0 \cat{B}$ to be abelian and
the action $\rho:\pi_0 \cat{B} \to \Aut(\pi_1 \cat{B})$ to be trivial.

\begin{thm}{\cite{eilenberg1954groups}}
\label{thm:EMquads}
    Equivalence classes of discrete braided categorical groups 
    $\cat{B}$ with
    $\pi_0 \cat{B} =A, \pi_1 \cat{B} =B$ 
    are in one-to-one correspondence with 
    quadratic forms $q:A \to B$. 
    Under this correspondence, $q(a)$ is the self-braiding
    $\beta_{a,a} \in \End(a) \iso B$
    of $a \in A$.
\end{thm}
A quadratic form is a map $q:A \to B$ such that $q(n \cdot a) = n^2 \cdot q(a)$ 
for all $n \in \IN, a \in A$, and the associated form 
\begin{align*}
    \sigma_q:A \times A &\to B \\
                  (a,b) &\mapsto q(a+b)/\left(q(a)q(b)\right)
\end{align*}
is bilinear.
We conjecture that Theorem~\ref{thm:EMquads} also holds true in the smooth case. We will only show 
one half of this statement, namely that
braided 2-groups are captured by their quadratic form.
Denote the bicategory of smooth braided 2-groups with $\pi_0 = A$, $\pi_1 = B$
for two abelian Lie groups $A,B$ by $\mathbf{B2G}(A,B)$.
It is easy to check that the quadratic form associated to any such
braided 2-group is smooth, and that the Baer sum of braided 2-groups
corresponds to pointwise product of the associated quadratic forms.

\begin{lem}
    The homomorphism 
    \begin{equation*}
        \mathbf{B2G}(A,B) \to \mathrm{Quad}^\mathrm{sm}(A,B)
    \end{equation*}
    that sends a braided 2-group to its quadratic form
    is injective.
    We assume the group of connected components of $A$ is 
    finitely generated.
\end{lem}
\begin{proof}
    We need to show that there is only one (braided) equivalence class
    of braided 2-groups that give the \emph{trivial} quadratic form 
    $q:A \to B$.
    The associator
    is detected by degree three Segal-Mitchison 
    cohomology $\H^3(\B A,B)$~\cite[Thm 99]{schommer2011central}.
    This cohomology group \emph{injects} into the corresponding 
    cohomology group of discrete groups~\cite{wagemann2015cocycle}.
    By Theorem~\ref{thm:EMquads}, the associator of the underlying
    discrete 2-group is trivial up to equivalence.
    As a result, we can assume that the associator of the smooth
    2-group is trivial, so we are considering braidings on
    $A \times [\ast/B]$. The hexagon equations demand that 
    the braiding be a bilinear map $\beta:A \tensor A \to B$.
    As $q(a)=\beta(a,a)$ is trivial, $\beta$ is alternating.
    A braiding on $A \times [\ast/B]$ is equivalent to the trivial braiding
    if there is a monoidal functor whose underlying functor is the identity
    and whose monoidal structure 2-cell $\eta:A \times A \to B$ makes the square
    \begin{equation*}
        \begin{tikzcd}
            a \tensor a' \rar{\eta(a,a^\prime)} \dar[swap]{\beta(a,a^\prime)} &
            a \tensor a' \dar{\beta^\mathrm{triv}(a,a^\prime)=1} \\
            a' \tensor a \rar{\eta(a^\prime,a)} &
            a' \tensor a
        \end{tikzcd}
    \end{equation*}
    commute for all $a,a' \in A$.
    Monoidality of the functor demands that $\eta$ is a bilinear map.
    Hence, $\beta$ is equivalent to the trivial braiding if there is a
    bilinear map $\eta:A \times A \to B$ such that 
    $\beta(a,a^\prime) = \eta(a,a^\prime) - \eta(a^\prime,a)$.

    It remains to show that any alternating bilinear form $\beta$ admits such a 
    trivialisation $\eta$.
    The group $A$ splits as a product of a discrete group $K$ and an abelian 
    Lie group $H$.
    Let $\mathfrak{h}$ be the Lie algebra of $H$, and $x,x' \in \mathfrak{h}$, $k,k' \in K$.
    The value of a bilinear form on $(k \cdot \exp(\alpha x),k' \cdot \exp(\alpha' x')) \in A^{\times 2} = {(K \times H)}^{\times 2}$ is
    completely determined by its values on pairs of elements of $k,k',x,x'$ (where evaluating
    on a Lie algebra element takes the derivative).
    Pick a set of generators $I_K$ for $K$ and $I_\mathfrak{h}$ of the Lie algebra 
    $\mathfrak{h}$ of $H$, as well as an ordering on $I = I_K \cup I_\mathfrak{h}$.
    This allows encoding bilinear forms completely in an $I \times I$ matrix whose
    entries are given by evaluating on the generators corresponding to the
    row and column.
    As $\beta$ is alternating, it is represented by an antisymmetric matrix. 
    We may simply pick $\eta$ to be its upper triangular half.
\end{proof}

Let $\mathrm{Quad}^\mathrm{sm}(A,B) \to \H^3(\B A,B)$ be the homorphism extracting
the monoidal structure of a braided categorical group 
from the quadratic form.
There exist explicit formulae that recover an associator from the data of a quadratic form, 
see~\cite{quinn1998group,braunling2020quinn}. Taking monoidal equivalence classes 
gives the above map.
Its kernel is the group of braidings for
the 2-group with trivial associator. 
When the associator is trivial, the hexagon equations
reduce to character equations of the map $\beta:A \times A \to B$ 
in each variable. Thus braidings are exactly
bilinear maps $A \tensor A \to B$.
The image of the map $\mathrm{Quad}^\mathrm{sm}(A,B) \to \H^3(\B A,B)$ is called
\emph{soft} cohomology $\H^3_\mathrm{soft}(\B A,B) \subset \H^3(\B A,B)$
in~\cite{davydov2018third}. 
Soft cohomology is the group of (equivalence classes of) associators 
that can be part of a \emph{braided} monoidal structure.
In summary, we get a sequence of groups, exact at $\mathrm{Quad}^\mathrm{sm}(A,B)$:
\begin{equation*}
        \mathrm{Bilin}(A,B) \into
        \mathrm{Quad}^\mathrm{sm}(A,B) \onto
        \H^3_\mathrm{soft}(\B A,B) \into
        \H^3(\B A,B).
\end{equation*}
\begin{example}
    We work out the case $A=\IZ/n$, $B=U(1)$.
    The relevant cohomology group 
    is $\H^3(\B \IZ/n,U(1)) \simeq \H^4(\B \IZ/n,\IZ) = \IZ/n$.
    The group of bilinear forms  $\IZ/n \tensor \IZ/n = \IZ/n \to U(1)$ 
    is $\mathrm{Bilin}(\IZ/n,U(1)) = \IZ/n$, 
    generated by a primitive n-th root of unity.
    Quadratic forms $q:\IZ/n \to U(1)$ are determined by the value on a
    generator: $q(k) = {q(1)}^{k^2}$.
    This defines a quadratic form iff 
    $q(1) \in \IZ/(2n,n^2)$, so
    \begin{equation*}
        \mathrm{Quad}(\IZ/n,U(1)) = 
        \begin{cases}
            \IZ/n & n \text{ odd} \\
            \IZ/2n & n \text{ even.} \\
        \end{cases}
    \end{equation*}
    The exact sequence of groups introduced above implies
    \begin{equation*}
        \H^3_\mathrm{soft}(\B \IZ/n,U(1))= 
        \begin{cases}
            0     & n \text{ odd} \\
            \IZ/2 & n \text{ even}.
        \end{cases}
    \end{equation*}
    For odd $n$, only the 2-group with trivial
    associator admits a braiding, while for even $n$, the
    2-group corresponding to $[n/2] \in \H^3(\B \IZ/n,U(1)) = \IZ/n$ 
    also does.
\end{example}
\begin{notation}
    \label{not:braidedZ2}
    A quadratic form $\IZ/2 \to U(1)$ must send the generator of $\IZ/2$
    to a fourth root of unity, $q(1)=\ii^k$.
    We borrow notation from the world of tensor categories to denote the
    corresponding braided 2-group $(\IZ/2,q)$
    as below. 
    \begin{table}[H]
        \centering
        \begin{tabular}{c|cccc}
            $q(1)$ & 1                     & $\ii$          
                   & $-1$           & $-\ii$                    \\ \hline
            $(\IZ/2,q)$ & $\Vec_{\IZ/2}^\times$ & $\Semi^\times$ 
                        & $\sVec^\times$ & $\bSemi^\times$
         \end{tabular}
    \end{table} 
    The braided tensor category we denote by $\Semi$ is known as the \emph{semion category}.
\end{notation}
            
\section{The centres of categorical tori}
\label{sec:catTori}
A \emph{categorical torus} $\mscr{T}$~\cite{ganter2018categorical} 
is a central extension
\begin{equation*}
    [\ast/U(1)] \to \mscr{T} \to T
\end{equation*}
of a (compact) torus $T$ by $[\ast/U(1)]$.
They are classified by $\H^3(\B T,U(1)) \simeq \H^4(\B T,\IZ)$.
Let $T$ be a torus, 
$\Lambda = \Hom(T, U(1))$ its group of characters,
$\Pi = \Lambda^\vee = \pi_1 T$ its group of cocharacters and 
$\mathfrak{t} = \mathrm{Lie}(T)$ its Lie algebra, identified as the universal cover
of $T$ via 
\begin{tikzcd}
	\Pi & {\mathfrak{t}} & T.
	\arrow[hook, from=1-1, to=1-2]
	\arrow["\exp", two heads, from=1-2, to=1-3]
\end{tikzcd}
The topological space $|\B T|$ is homotopy equivalent to an
$r$-fold product of ${\IC\mathbb{P}}^\infty$'s, where $r$ is the 
rank of the torus $T$. The cohomology
ring of $\B T$ is naturally identified with 
$\H^\ast(\B T, \IZ)=\H^\ast(|\B T|,\IZ)=\mathrm{Sym}^\ast(\Lambda)$,
where $\Lambda$ is placed in degree 2.
We identify the group $\H^4(\B T,\IZ) = \mathrm{Sym}^2(\Lambda)$ 
with the group of symmetric bilinear forms 
$I:\mathfrak{t} \times \mathfrak{t} \to \IR$
such that for all $\pi \in \Pi$, $I(\pi,\pi)\in 2\IZ$.
An element $\lambda \in \Lambda = \Hom(T, U(1))$ induces a map of Lie algebras
$D_e\lambda: \mathfrak{t} \to \IR$. Then we send
$\lambda_1 \tensor \lambda_2 \in \Lambda^{\tensor 2}$ to the symmetric 
bilinear form $I:(x, y) \mapsto D_e\lambda_1(x) \cdot D_e\lambda_2(y) + D_e\lambda_2(x) \cdot D_e\lambda_1(y)$.

Given a class $I \in \H^4(\B T, \IZ)$, we will use $\tau$ to denote the induced map 
$\Pi \to \Lambda = \Pi^\vee$, given by $\tau(\pi)=I(\pi,-)$.
We further pick a (not necessarily symmetric) bilinear form
$J$ on $\mathfrak{t}$ such that $J$ restricts to 
a $\IZ$-valued form on $\Pi$, and $I =-(J + J^t)$,
where $J^t$ denotes the transpose of $J$.
We recall a construction of the categorical torus $T_J$ 
given in~\cite{ganter2018categorical}. It is a strict smooth 2-group
(Example~\ref{ex:strict2Gps}). 
The underlying Lie groupoid is 
\begin{equation*}
    \mathfrak{t} \ltimes (\Pi \times U(1)) \rightrightarrows \mathfrak{t},
\end{equation*}
with arrows $(x,\pi,w) = \left(x \xto{w} x+\pi\right)$. 
Composition and tensor product are given by 
\begin{align*}
    \left(x+\pi \xto{w^\prime} x+\pi+\pi^\prime\right) \comp 
    \left(x \xto{w} x+\pi\right) &= 
    \left(x \xto{w w^\prime} x+\pi+\pi^\prime\right) \\
    \left(x \xto{w} x+\pi\right) \tensor
    \left(x^\prime \xto{w^\prime} x+\pi^\prime\right) &=
    \left(x+x^\prime \xto{w w^\prime \exp(J(\pi,x^\prime))} 
    x+x^\prime+\pi+\pi^\prime\right).
\end{align*}
The associator and unitor cells are trivial.
This categorical torus is classified up to equivalence by the symmetric
bilinear form $I = -(J + J^t) \in \H^4(\B G,\IZ)$~\cite[Thm 4.1]{ganter2018categorical}.
The paper~\cite{freed2010topological} contains a sketch of calculation of the 
smooth Drinfel'd centre of $T_J$.
We complement this with a proof based on the construction above.
\begin{prop}
\label{prop:DZofTorus}
    The Drinfel'd centre of $T_J$ has underlying Lie groupoid
    \begin{equation*}
        \DZ T_J = \left((\mathfrak{t} \dirSum \Lambda)
        \ltimes (\Pi\times U(1)\right)  
        \rightrightarrows \mathfrak{t} \dirSum \Lambda,
    \end{equation*}
    with arrows $(x, \lambda, \pi, w):(x, \lambda) \to 
    (x+\pi, \lambda + \tau(\pi))$.
    The braiding is given by
    \begin{equation*}
        \beta_{[x, \lambda],[x^\prime,\lambda^\prime]} = 
        \lambda(x^\prime) \exp(J(x^\prime,x)).
    \end{equation*}
\end{prop}
\begin{proof}
    A half-braiding on an object $x \in \mathfrak{t}$ is a 2-cell
    $\gamma:x \tensor - \to - \tensor x$ (subject to the hexagon equation). 
    Tensoring with $x$ on either side is a smooth functor $T_J \to T_J$. 
    Hence each such 2-cell $\gamma$ is 
    represented by a smooth natural transformation:
    a smooth map 
    $\gamma: \mathfrak{t} \to \mathfrak{t} \ltimes (\Pi \times U(1))$, 
    which sends $y \in \mathfrak{t}$ to an endomorphism of $x+y=y+x$.
    In the absence of associators, the hexagon equation simplifies to 
    the character equation $\gamma_{y+z}=\gamma_z \gamma_y$.
    The condition that $\gamma$ be natural gives it
    the form 
    \[
        \gamma(y) = \lambda(y) \cdot \exp(J(y,x)), 
    \]
    where $\lambda \in \Lambda$ is a smooth character of $T$.
    This allows a parameterisation of objects of $\DZ T_J(\ast)$ by pairs 
    $(x,\lambda) \in \mathfrak{t} \dirSum \Lambda$.
    
    
    Morphisms $(x,\lambda) \to (x^\prime,\lambda^\prime)$ are 
    morphisms $x \to x^\prime$ which are compatible with the braidings.
    The morphism $(x,\pi,w)\in \mathfrak{t} \ltimes (\Pi \times U(1))$ 
    is a morphism of centre pieces $(x,\lambda) \to (x+\pi, \lambda^\prime)$
    precisely when 
    \[
    \lambda^\prime(y) = \lambda(y) \cdot \exp(-J(y,\pi)-J(\pi,y)) =
    \left(\lambda + \tau(\pi)\right)(y).
    \]
    for all $y \in \mathfrak{t}$.
    
    
    So far, we have calculated $\DZ T_J(\ast)$. To deduce the 
    smooth structure, note that any object $\tilde{x} \in T_J(V)$ 
    can be represented by a smooth map $\tilde{x}:V \to \mathfrak{t}$,
    and the corresponding half-braidings are represented by smooth maps
    $\tilde{\gamma}:V \times \mathfrak{t} \to \mathfrak{t} \ltimes (\Pi \times U(1))$.
    Naturality and the hexagon equation can now be checked pointwise.
    
    
    The braiding on the centre can be computed in $\DZ(T_J(\ast)) $, where 
    it takes the usual form
    $\beta_{(x,\gamma),(x^\prime,\gamma^\prime)} = 
    \gamma(x^\prime)$~\cite[Ch 8.5]{etingof2016tensor}.
    This completes the proof.
\end{proof}

As mandated by Lemma~\ref{lem:pi1ZCispi1C}, $\pi_1 \DZ T_J = U(1)$. 
The object group of the centre is 
\begin{equation*}
    \pi_0 \DZ T_J = \frac{\mathfrak{t}\dirSum \Lambda}{\Pi},
\end{equation*}
with $\Pi \ni \pi \mapsto (\pi,\tau(\pi)) \in \mathfrak{t}\dirSum \Lambda$.
We end this section by computing the \emph{maximal compact subgroup} of $\pi_0 \DZ T_J$.
The cohomology class corresponding to the symmetric bilinear form
$I =-(J + J^t)$ induces the maps 
$\tau: \Pi \to \Lambda$ and 
$\tau_\IR \define \tau \tensor \IR: \mathfrak{t} \to \mathfrak{t}^\ast$.
The case where $\tau_\IR$ is an isomorphism (such levels are referred to as
\emph{non-degenerate}) was already analysed  in~\cite{freed2010topological}.
The map $(\mathfrak{t} \dirSum \Lambda)/\Pi \to \Lambda/\Pi$ is split by 
$s:\lambda \mapsto (\tau_\IR^{-1}\lambda,\lambda)$. This furnishes an isomorphism 
$\pi_0 \DZ T_J \eqto \mathfrak{t} \dirSum (\Lambda/\Pi)$. Note that $\Lambda/\Pi$ is
finite, so $\pi_0 \DZ T_J$ is a direct sum of $\IR^{\rk{\tau_\IR}}$ 
with a finite abelian group.

For general level, we pick a splitting $\Pi \iso \Pi_\mker \dirSum \Pi_\coim$, 
where $\Pi_\mker \define \ker{\tau}$. We tensor with $\IR$ to obtain 
$\mathfrak{t} \iso \mathfrak{t}_\mker \dirSum \mathfrak{t}_\coim$.
Now $\tau_\IR$ restricts to an isomorphism 
$\tau_\im:\mathfrak{t}_\coim \eqto \im\ \tau_\IR \subset \mathfrak{t}^\ast$. 
We denote the intersection ${(\im\ \tau_\IR)} \cap {\Lambda}$ by $\Lambda_{\im_\IR}$.
Lastly, we pick a decomposition 
$\Lambda \iso \Lambda_{\im_\IR} \dirSum \Lambda_{\coker_\IR}$.
These isomorphisms assemble into 
\begin{equation*}
    \pi_0 \DZ T_J \eqto \Lambda_{\coker_\IR} \dirSum 
    \frac{\mathfrak{t}_\mker}{\Pi_\mker} \dirSum 
    \frac{\mathfrak{t}_\coim \dirSum \Lambda_{\im_\IR}}{\Pi_\coim} \iso 
    \Lambda_{\coker_\IR} \dirSum T_\mker \dirSum \mathfrak{t}_\coim 
    \dirSum \Lambda_{\im_\IR}/\Pi_\coim,
\end{equation*}
where the second isomorphism uses that $\tau$ is non-degenerate when 
viewed as a map $\Pi_{\coim} \to \Lambda_{\im_\IR}$.
As before, $\Lambda_{\im_\IR}/\Pi_\coim$ is a finite group.
Both $\Lambda_{\coker_\IR}$ and $\mathfrak{t}_\coim$ are free, and 
$T_\mker \define \mathfrak{t}_\mker/\Pi_\mker$ is a compact torus.
One may now read off the 
maximal compact subgroup as $T_\mker \dirSum \Lambda_\im/\Pi_\coim$.
We are justified in calling this \emph{the} maximal compact subgroup: 
any element of the free subgroup 
$\Lambda_{\coker_\IR} \dirSum \mathfrak{t}_\coim$ generates 
a non-compact group, so all compact subgroups of 
$\pi_0 \DZ T_J$ must intersect trivially with it ---
the maximal compact subgroup we computed above 
is the orthogonal complement of 
$\Lambda_{\coker_\IR} \dirSum \mathfrak{t}_\coim$.
Using the explicit maps above, it is straightforward to compute the 
induced braiding on it.
We call the resulting braided categorical group the maximal compact sub-2-group.
\begin{prop}
\label{prop:maxCpctSubGp}
    The maximal compact sub-2-group of $\DZ T_J$ is the braided 2-group 
    with 
    \begin{equation*}
        \pi_0\operatorname{Cpct}{\DZ T_J} = 
        T_{\mker} \dirSum \Lambda_\im / \Pi_\coim,
    \end{equation*}
    $\pi_1 \DZ T_J = U(1)$, and braiding encoded by the quadratic form
    \begin{equation*}
        \bar{q}([t_\mker, \lambda])=\lambda(t_\mker + \tau_\im^{-1}\lambda) 
        \exp(J(\tau_\im^{-1}\lambda,\tau_\im^{-1}\lambda)).
    \end{equation*}
\end{prop}
\begin{proof}
    This is a straightforward computation. The component 
    $T_\mker$
    does not contribute in the exponential because 
    \begin{align*}
    J(t_\mker,t_\mker)=\tfrac{1}{2}\tau(t_\mker)(t_\mker)=0
    &&
    \text{ and }
    &&
    J(t_\mker,\tau_\im^{-1}\lambda)+J(\tau_\im^{-1}\lambda,t_\mker)=0.
    &
    \qedhere
    \end{align*}
\end{proof}

We record one important feature of this subgroup.
\begin{lemma}
\label{lem:maxCpctInjects}
    The map
    \begin{equation*}
        u:\pi_0 \operatorname{Cpct}{\DZ T_J} \to T
    \end{equation*}
    that forgets the half-braiding is \emph{injective}.
\end{lemma}
\begin{proof}
    The decomposition of $\mathfrak{t}$ 
    descends to a decomposition of the torus $T$.
    Then $u$ is the direct sum
    of injective maps
    \begin{equation*}
        \begin{tikzcd}[column sep=huge]
            \pi_0 \operatorname{Cpct}{\DZ T_J} = 
            T_{\mker} \dirSum \Lambda_\im/\Pi_\coim 
            \rar["\id \dirSum \tau_\im^{-1}",hook] & 
            T_\mker \dirSum \mathfrak{t}_\coim/\Pi_\coim = T_\mker \dirSum T_\coim.
        \end{tikzcd}
        \qedhere
    \end{equation*}
\end{proof}

\section{The centres of String 2-groups}
\label{sec:stringGrps}
In this section, we compute the centre $\DZ$ for the String groups $G_k$.
We do this using obstruction theory.
In particular, we establish an exact sequence of groups in 
Proposition~\ref{prop:exactSequence}, which relates 
the group $\pi_0 \DZ G_k$ to the ordinary centre $Z(G)$ and 
the group cohomology of $G$.
When $G$ is simply-connected, this sequence allow us to deduce that 
every element of the ordinary centre $z \in Z(G)$ admits a 
unique lift to the Drinfel'd centre $\DZ G_k$. 
This lift restricts to an element in $\DZ T_k$, 
where $T_k$ is a maximal 2-torus for $G_k$. 
The computations of Section~\ref{sec:catTori} allow us to
deduce the resulting braided structure.
We then treat the case of non-simply-connected $G$
by picking a simply-connected 
covering group $\pi:\tilde{G} \onto G$ and checking which 
centre pieces in $\DZ \tilde{G}_{\pi^\ast k}$ descend to 
centre pieces of $\DZ G_k$.

Finally, we show the 
Drinfel'd centre $\DZ G_k$ agrees with 
the invertible part of $\Rep^k LG$ (if $G$ is semisimple and there is 
no factor of $E_8$ at level 2).

\subsection*{Simply-connected Lie groups}
Let $G$ be a compact simple simply-connected Lie group,
pick a maximal torus $T \into G$, and let $\mathfrak{t}$,
$\mathfrak{t}^\ast$, $\Lambda$, $\Pi$ be as in Section~\ref{sec:catTori}.
The Lie algebra $\mathfrak{t} \subset \mathfrak{g} = \Lie(G)$ 
is also called a \emph{Cartan algebra} for $G$. 
The \emph{Weyl group} $W = N(T)/T$, where $N(T)$ denotes the normaliser of $T$ in $G$, 
acts on $T$ by conjugation.

Recall from Section~\ref{sec:catTori} that $\H^4(\B T, \IZ)$ is identified with 
the group of $\IR$-valued symmetric bilinear forms $I$ on $\mathfrak{t}$ satisfying 
$I(\pi,\pi) \in 2 \IZ$ for all $\pi \in \Pi \subset \mathfrak{t}$.
Borel~\cite{borel1953cohomologie,borel1954homologie} identified 
$\H^4(\B G, \IZ)$ as the Weyl-invariant part 
$\H^4(\B G, \IZ) = \H^4(\B T, \IZ)^W$, see~\cite{toda1987cohomology} for 
a review. 
We identify $\H^4(\B T, \IZ)^W$ with the group of those
inner products $I$ as above that are invariant under the $W$-action on
$\mathfrak{t}$. 
This identification holds true for general compact connected Lie 
groups~\cite[Thm 6]{henriques2017classification}.

For $G$ compact simple and simply-connected, 
$\H^4(\B G,\IZ) = \H^4(\B T,\IZ)^W \iso \IZ$
has a distinguished generator: the basic positive-definite Weyl-invariant
inner product $I:\mathfrak{t} \times \mathfrak{t} \to \IR$,
normalised such that short coroots have norm squared 2.
We denote by $k \in \IZ$ the cohomology class corresponding 
to $k \cdot I \in \H^4(\B G, \IZ)$. 
As before, we denote the map induced by $k \cdot I$ on $\Pi$ by
$\tau: \Pi \to \Lambda$ --- recall it sends $\pi \mapsto k \cdot I(\pi,-)$.
Every non-zero cohomology class $k$ induces an isomorphism
$\tau_\IR = \tau \tensor \IR: \mathfrak{t} \to \mathfrak{t}^\ast$.

One may pull back the extension of $G$ by $[\ast/U(1)]$ 
along the inclusion of a maximal torus $T \into G$
to obtain a maximal 2-torus.
In~\cite{ganter2018categorical}, Ganter shows that the
maximal 2-torus of $G_k$ is the categorical torus $T_{J_k}$
associated to a (non-symmetric) bilinear form 
$J_k:\mathfrak{t} \times \mathfrak{t} \to \IR$ 
such that $J_k+J_k^t = -k \cdot I$.

\begin{prop}
\label{prop:exactSequence}
    Let $H_\omega$ denote the central extension of $H$ by $[\ast/A]$ ($A$ an abelian Lie 
    group) corresponding to $\omega \in \H^3(\B H,A)$.
    Then there is an exact sequence
    \begin{equation*}
        0 \to \H^1(\B H, A) \to \pi_0 \DZ H_\omega \to Z(H) 
        \to \H^2(\B H, A).
    \end{equation*}
\end{prop}
\begin{proof}
    Each centre piece $(g,\gamma) \in \DZ H_\omega(\ast)$ must 
    satisfy $g \in Z(H)$, otherwise $g x \not \iso x g$ for 
    some $x \in H$.
    In the discrete case, one can work with a skeletal representative of $H_\omega$.
    The hexagon equation evaluated at $(x,y) \in G\times G$ is then
    \begin{equation*}
        (d\gamma)(x,y) \define \frac{\gamma(xy)}{\gamma(x)\gamma(y)} =
        \frac{\omega(x,g,y)}{\omega(g,x,y)\omega(x,y,g)}
        \eqqcolon \omega(x,y|g),
    \end{equation*}
    so $\gamma$ is a 1-cochain whose coboundary is $\omega(-,-|g)$.
    It is straightforward to check that $\omega(-,-|g)$ is indeed a 2-cocycle.
    The map $Z(H) \to \H^2(\B H,A)$ assigns to each element
    the equivalence class of the corresponding cocycle:
    $g \mapsto [\omega(-,-|g)]$.
    The element $g$ admits a half-braiding precisely
    if $\omega(-,-|g)$ is a coboundary, which proves exactness 
    at $Z(H)$.
    
    The kernel of the map $\pi_0 \DZ H_\omega \to Z(H)$ is 
    the group of half-braidings for the identity element of $H$.
    The associator $\omega$ can be chosen to be trivial whenever 
    at least one of the 
    entries is the identity, and the hexagon equation
    becomes the equation of an $A$-valued character on $H$. These
    are precisely the elements of $\H^1(\B H, A)$.
    
    
    We now port the above proof to the smooth case.  
    Recall that the Lie groupoid modelling $H_\omega$ is of
    the form $L \rightrightarrows Y$, where $Y \onto H$ is a surjective
    submersion and $L \to Y^{[2]}$ is an $A$-bundle.
    We pick a cover $\kappa: V \onto Y$ such that the line bundle $L$ 
    trivialises over $V^{[2]}$,
    and replace $H_\omega$ by the equivalent groupoid 
    $\kappa^\ast H_\omega = V^{[2]} \times A 
    \rightrightarrows V$.
    Then we pick a cover $\pi:W \to V \times V \to Y \times Y$ 
    such that all six functors/vertices
    in the hexagon equation for $\gamma$ 
    (ie. $(g \tensor -) \tensor -, g \tensor (- \tensor -), 
    (- \tensor -) \tensor g \ldots$) are representable by smooth 
    functors $\pi^\ast (H_\omega \times H_\omega) \to \kappa^\ast H_\omega$.
    Each 2-morphism/edge in the hexagon equation 
    ($\omega(z,-,-),\gamma(-\tensor-),\ldots$)
    is then represented by
    a smooth natural transformation.
    Each pair of functors $F_i, F_j$ gives a map
    $f_{ij}:W \to V^{[2]}$, and a smooth natural transformation $F_i \Rightarrow F_j$
    is a section of the pullback bundle $f_{ij}^\ast (V^{[2]} \times A)$.
    Under the choices we made, these bundles are all trivial. 
    The hexagon axiom reduces to the same equation as above,
    $d\gamma = \omega(-,-|g)$,
    except it is now an equation in Segal-Mitchison cohomology.
    The pair $(V \to H, W \to H \times H)$ forms the first two
    steps of a simplicial cover of $\B H$.
    The maps $\gamma:V \to A, \omega(-,-|g):W \to A$ 
    represent Segal-Mitchison cochains.
\end{proof}

\begin{cor}
\label{cor:pi0isZG}
    Each element $z \in Z(G)$ admits a unique 
    half-braiding over $G_k$:
    \begin{equation*}
        \pi_0\DZ G_k = Z(G).
    \end{equation*}
\end{cor}
\begin{proof}
    Compactness of $G$ implies $\H^\ast(\B G,U(1)) 
    \simeq \H^{\ast+1}(\B G, \IZ)$. 
    The connectivity assumptions further imply 
    $\H^2(\B G,\IZ)=\H^3(\B G,\IZ)=0$ (see Section~\ref{sec:2Grps}).
    The exact sequence of Proposition~\ref{prop:exactSequence} 
    shortens to an isomorphism.
\end{proof}

Every half-braiding for $z \in Z(G)$ over $G_k$
restricts to a half-braiding of $z$ over $T_{J_k}$.
We thus get a restriction functor $r:\DZ G_k \to \DZ T_{J_k}$. 
To describe this restriction functor explicitly, we recall how the 
centre $Z(G)$ of 
a Lie group lifts to $\mathfrak{t}$ --- see Chapter 13 of~\cite{hall2015lie}
for proofs of the following facts. 
The centre $Z(G)$ includes into any maximal torus $T$ of $G$. 
It lifts to $\mathfrak{t}$ as the dual $\Phi^\vee$ of the 
\emph{root lattice} $\Phi \subset \mathfrak{t}^\ast$ of $G$.
For simply-connected $G$, the lattice $\Phi^\vee$ agrees with the
\emph{coweight lattice} (see e.g.~\cite{kirillov2005compact}),
and thus elements of $Z(G)$ lift to coweights in 
$\Phi^\vee \subset \mathfrak{t}$. The centre of a simply-connected 
compact Lie group may be computed
from the coweight and cocharacter lattice as $Z(G) = \Phi^\vee / \Pi$.

\begin{thm}
\label{thm:DrinfeldCentreOfGk}
    The Drinfel'd centre of $G_k$ is the 
    braided categorical group specified by $\pi_0 \DZ G_k = Z(G)$, 
    $\pi_1 \DZ G_k = U(1)$ and the quadratic form
    \begin{align*}
        q: Z(G) &\to U(1) \\
        z &\mapsto \exp\left(\tfrac{k}{2}I(\bar{z},\bar{z})\right),
    \end{align*}
    where $\bar{z}$ denotes any lift of $z \in Z(G)$ to $\mathfrak{t}$.
\end{thm}
\begin{proof}
    As $\pi_0 \DZ G_k = Z(G)$ is finite, the functor $r:\DZ G_k \to \DZ T_{J_k}$
    must land in the maximal compact subgroup $\operatorname{Cpct}{\DZ T_{J_k}}$,
    computed in Proposition~\ref{prop:maxCpctSubGp}.
    It fits into the commutative diagram
    \begin{equation*}
        \begin{tikzcd}
            \pi_0\DZ G_k \arrow[r, "r"] \arrow[d, "\simeq"] & 
            \pi_0 \operatorname{Cpct}{\DZ T_{J_k}} \arrow[d, hook, "u"] \\
            Z(G) \arrow[ru, "r_0" swap] \arrow[r, hook]                 & T,
        \end{tikzcd}
    \end{equation*}
    where the left hand map is an isomorphism by Corollary~\ref{cor:pi0isZG},
    and the right hand map $u$ is injective by Lemma~\ref{lem:maxCpctInjects}.
    The map $r_0:Z(G) \to \pi_0 \operatorname{Cpct}{\DZ T_{J_k}}$ is uniquely 
    fixed by the requirement that the bottom right triangle commute.
    
    For $k=0$, the map $u$ is the identity on $T$. The braiding is trivial on all of $T$, 
    and restricts to the trivial braiding on $Z(G)$.
    For $k \neq 0$, the maximal compact subgroup is $\Lambda/\Pi$, $u$ is equal to 
    $\tau_\IR^{-1}: \Lambda/\Pi \to \mathfrak{t}/\Pi = T$,
    and $r_0$ is the section 
    \begin{equation*}
        \begin{tikzcd}
            Z(G) =\Phi^\vee/\Pi \rar[hook,"\tau/\Pi"] 
            & \Lambda/\Pi.
        \end{tikzcd}
    \end{equation*}
    Denote a lift of $z \in Z(G)$ to $\Phi^\vee$ by $\bar{z}$.
    The quadratic form $\bar{q}$ computed in 
    Proposition~\ref{prop:maxCpctSubGp} pulls back to
    \begin{equation*}
        q(z) = \bar{q}(\tau\bar{z}) = \tau(\bar{z})(\bar{z}) 
        \exp J_k(\bar{z}, \bar{z})
        = \exp\left(k\cdot I(\bar{z},\bar{z})+J_k(\bar{z},\bar{z})\right)
        = \exp\left(\tfrac{k}{2}I(\bar{z},\bar{z})\right).
        \qedhere
    \end{equation*}
\end{proof}

Each lift $\bar{z}$ in this formula is a coweight. 
The norm of a coweight may be computed as the norm of 
the corresponding weight of the \emph{dual root datum}.
In the realm of compact simple simply-connected Lie groups, 
dualising root data simply exchanges the odd-dimensional Spin groups 
$B_n = \Spin(2n+1)$ and the symplectic Lie groups $C_n = \mathrm{Sp}(2n)$. 
All other groups are fixed by this operation.
The norm is computed using the inner product on the dual of the Cartan 
of the dual root datum.
It is normalised such that short roots have length squared 2.
The values of the length squared of weights under this product 
can be read off from the explicit expansion 
for weights in terms of roots given in~\cite{bourbaki1994lie}.
In Table~\ref{tab:resultsTable}, we list the results of
this computation. 

\begin{example}
\label{ex:DrinfeldCentreOfSU2levelk}
    The Drinfel'd centre of ${\SU(2)}_k$ is given by (see Notation~\ref{not:braidedZ2})
    \begin{equation*}
        \DZ {\SU(2)}_k = 
        \begin{cases}
            {\Vec_{\IZ/2}}^\times & k \equiv 0 \mod 4 \\
            {\Semi}^\times & k \equiv 1 \mod 4 \\
            {\sVec}^\times & k \equiv 2 \mod 4 \\
            {\bSemi}^\times & k \equiv 3 \mod 4.
        \end{cases}
    \end{equation*}
\end{example}

\subsection*{Non-simply-connected Lie groups}
Any compact connected Lie group $G$ fits into a short exact sequence
\begin{equation*}
    \begin{tikzcd}
        Z \rar[hook] & \tilde{G} \rar[two heads, "\pi"] & G,
    \end{tikzcd}
\end{equation*}
where the middle term is a product $\tilde{G} = T \times \Pi_i G_i$ of
a torus $T$ with
simply-connected simple Lie groups $G_i$,
and $Z \into Z(\tilde{G})$ is a finite central 
subgroup~\cite[Cor V.5.31]{mimura1991topology}.
The degree 4 cohomology of $\tilde{G}$ decomposes as
\begin{equation*}
    \H^4(\B \tilde{G},\IZ) = \H^4(\B T,\IZ) \dirSum \Pi_i \H^4(\B G_i,\IZ).
\end{equation*}
Hence any degree 4 cohomology class of $\tilde{G}$ can be represented by
a cocycle which is a product of cocycles pulled back from the individual factors. 

\begin{lemma}
\label{lem:DZofProduct}
    Let $H, H^\prime$ be a pair of Lie groups and $\omega, \omega^\prime$ cocycles 
    representing associators. Denote by $\bar{\omega} \define p_H^\ast \omega + 
    p_{H^\prime}^\ast \omega^\prime$ the product cocycle on $H \times H^\prime$.
    The centre $\DZ {(H \times H^\prime)}_{\bar{\omega}}$ is the braided abelian
    group $(\pi_0\DZ H_\omega \times \pi_0\DZ H^\prime_{\omega^\prime}, 
    \bar{q})$, with quadratic form
    \begin{equation*}
        \bar{q}: \pi_0 \DZ H_\omega \times 
        \pi_0 \DZ H^\prime_{\omega^\prime} \to U(1) =
        \pi_1 \DZ {(H \times H^\prime)}_{\bar{\omega}}
    \end{equation*}
    given by the pointwise product of the quadratic forms on $\DZ H_\omega$ and 
    $\DZ H^\prime_{\omega^\prime}$.
\end{lemma}
\begin{proof}
    We prove this in the discrete setup. The argument is carried over to the 
    smooth setting exactly as in the proof of Proposition~\ref{prop:exactSequence}.
    A centre piece for $(H \times H^\prime)_{\bar{\omega}}$ is 
    a tuple of elements $(h,h^\prime)$, equipped with a half-braiding
    $\gamma:H \times H^\prime \to U(1)$ satisfying
    the hexagon equation
    \begin{equation*}
        d \gamma = \bar{\omega}\left(-,-|(h,h^\prime)\right).
    \end{equation*}
    We will now show that every such half-braiding is a product of half-braidings
    in $H_\omega$ and $H^\prime_{\omega^\prime}$ and vice versa.
    The cocycle $\bar{\omega}$ splits as a product of $\omega$ and $\omega^\prime$, 
    and thus so does $\bar{\omega}(-,-|(h,h^\prime))$. The hexagon 
    equation implies that the half-braiding splits as a product
    \begin{equation*}
        \gamma\left((x,x^\prime)\right) = 
        \gamma\left((x,e)\right) \gamma\left((e,x^\prime)\right)
        \eqqcolon \gamma_h(x) \gamma_{h^\prime}(x^\prime).
    \end{equation*}
    The hexagon equation for $\gamma$ is now equivalent to the hexagon
    equations $d\gamma_h=\omega(-,-|h)$ and
    $d\gamma_{h^\prime}=\omega^\prime(-,-|h^\prime)$.
    Hence $\pi_0 \DZ {(H \times H^\prime)}_{\bar{\omega}} = 
    \pi_0\DZ H_\omega \times \pi_0\DZ H^\prime_{\omega^\prime}$ is 
    indeed the product. 
    The quadratic form is given by the pointwise product
    because the half-braidings are. 
\end{proof}

The central extension $G_k$ of $G$ corresponding to $k \in \H^4(\B G,\IZ)$
pulls back to an extension $\tilde{G}_{\pi^\ast k}$, where $\pi^\ast k = 
(J,\{k_i\}) \in \H^4(\B T,\IZ) \dirSum \Pi_i \H^4(\B G_i,\IZ)$.
Its Drinfel'd centre $\DZ \tilde{G}_{\pi^\ast k}$ is the 
braided 2-group with 
\begin{equation*}
    \pi_0\DZ\tilde{G}_{\left(J,\{k_i\}\right)}=\DZ(T \times \Pi_i G_{i,k_i})=
    (\Lambda_T \dirSum \mathfrak{t}_T)/\Pi_T \times \Pi_i Z(G_{i,k_i}),
\end{equation*}
$\pi_1 \DZ \tilde{G}_{(J,\{k_i\})}= U(1)$, and quadratic form
given by the pointwise product $q = q_J \times \Pi_i q_{k_i}$.
The finite central subgroup $Z = \ker{\pi}$ lifts 
uniquely to $\pi_0\DZ \tilde{G}$: it must land in the maximal compact
subgroup, which injects into $\tilde{G}$ as
$\pi_0 \operatorname{Cpct}{\DZ T_J} \times \Pi_i Z(G_i) \into 
T \times \Pi_i Z(G_i)$
by Lemma~\ref{lem:maxCpctInjects}.

Choose a maximal torus $\tilde{T} \into \tilde{G}$, then
$\tilde{T}/Z \into \tilde{G}/Z = G$ is a maximal torus for $G$.
The Lie algebras of these two tori may both be identified with 
their common universal cover, which we denote $\tilde{\mathfrak{t}}$.
We write $\tilde{\Pi}$ for the fundamental group of $\tilde{T}$, 
embedded as a lattice in $\tilde{\mathfrak{t}}$.
The fundamental group $\tilde{\Pi}_Z$ of $\tilde{T}/Z$ is an extension
\begin{equation*}
    \tilde{\Pi} \into \tilde{\Pi}_Z \onto Z.
\end{equation*}
Considered as a lattice in $\tilde{\mathfrak{t}}$, it is the 
preimage of $Z \subset \tilde{T}$ in $\mathfrak{\tilde{t}}$:
\begin{equation*}
    \begin{tikzcd}
    \tilde{\Pi}_Z \arrow[d,two heads] \arrow[r,hook] & 
    \tilde{\mathfrak{t}} \arrow[d,two heads] \\
    Z \arrow[r,hook] & \tilde{T}
    \arrow[lu, "\lrcorner" very near end, phantom].
    \end{tikzcd}
\end{equation*}
The Weyl groups of $\tilde{G}$ and $G$ under these choices of maximal
tori are canonically isomorphic
and the Weyl-actions on $\tilde{\mathfrak{t}}$ agree.

Recall that we identified $\H^4(\B \tilde{G}, \IZ)$ with 
the group of symmetric bilinear forms $I$ on $\tilde{\mathfrak{t}}$
which are Weyl-invariant and satisfy $I(\tilde{\pi},\tilde{\pi}) \in 2\IZ$
for all $\tilde{\pi} \in \tilde{\Pi}$.
Under this identification with bilinear forms, the map 
$\pi^\ast: \H^4(\B G,\IZ) \to \H^4(\B \tilde{G},\IZ)$
restricts the bilinear form along $\pi$ 
(see also~\cite{henriques2017classification}).
The quotient map $\pi$ induces the identity on the Lie algebra 
$\tilde{\mathfrak{t}}$, and preserves the Weyl-action.
Hence, the image of the map $\pi^\ast:\H^4(\B G, \IZ) \to \H^4(\B \tilde{G}, \IZ)$ 
consists precisely of those symmetric bilinear forms on $\mathfrak{t}$ that 
satisfy the even integrality condition not only for $\tilde{\Pi}$, but for 
$\tilde{\Pi}_Z$. 

Now consider the unique lift of $Z$ to $\pi_0 \DZ \tilde{G}_{\pi^\ast k}$.
The quadratic form on $\pi_0 \DZ \tilde{G}_{\pi^\ast k}$ restricts to
\begin{align*}
    q: Z &\to U(1) \\
       z &\mapsto \exp \tfrac{1}{2}I(\bar{z},\bar{z}).
\end{align*}
It vanishes on $z \in Z$ if and only if $I(\bar{z},\bar{z}) \in 2\IZ$ for all 
lifts $\bar{z} \in \tilde{\mathfrak{t}}$. 
As discussed above, the preimage of $Z$ under $\pi$ is precisely $\tilde{\Pi}_Z$, 
so this is equivalent to the integrality condition for $I \in \H^4(\B G, \IZ)$ and
we arrive at the following:
\begin{lemma}
\label{lem:MapOnH4AndQuadraticForms}
    Let $Z \into \tilde{G} \onto G$ be as above.
    Then $\tilde{k} \in \H^4(\B \tilde{G}, \IZ)$ is in the 
    image of the map $\H^4(\B G, \IZ) \into \H^4(\B \tilde{G}, \IZ)$
    if and only if the quadratic form on $\pi_0 \DZ \tilde{G}_{\tilde{k}}$
    vanishes when restricted along the unique lift 
    $Z \into \pi_0 \DZ \tilde{G}_{\tilde{k}}$.
\end{lemma}
A description of the maps $\H^4(\B G, \IZ) \into \H^4(\B \tilde{G}, \IZ)$ 
for quotients $G=\tilde{G}/Z$ of simple 
simply-connected Lie groups may be found 
in~\cite[Table 1]{gawedzki2009polyakov}.

Consider the closed subgroup
\begin{equation*}
    Z^\bot \define \left\{ x \in \pi_0 \DZ \tilde{G} \mid
    q(x+z)=q(x), \forall z \in Z \right\}
\end{equation*}
of objects in $\DZ \tilde{G}$ that braid trivially with 
every element of $Z$.
By abuse of notation, we also denote by this the braided smooth sub-2-group of $\DZ \tilde{G}$ 
obtained by pulling back along the inclusion $Z^\bot \into \pi_0\DZ \tilde{G}$.
By Lemma~\ref{lem:MapOnH4AndQuadraticForms}, $q$ vanishes on $Z$, so in particular
$Z \subset Z^\bot$.
The quotient $Z^\bot/Z$ inherits a smooth structure, because $Z^\bot$ is an embedded 
Lie subgroup of $\pi_0 \DZ \tilde{G}$, and $Z$ is a normal subgroup of $Z^\bot$.
\begin{thm}
\label{thm:nonSimplyConn}
    The Drinfel'd centre of $G_k$ is the smooth braided categorical group
    \begin{equation*}
        \DZ G_k = (Z^\bot / Z, q\restriction_{Z^\bot}).
    \end{equation*}
\end{thm}
\begin{proof}
    Let $(a, \gamma) \in \DZ G_k$ be a centre piece for $G_k$. Pick a lift 
    $\bar{a} \in Z(\tilde{G})$
    of $a \in Z(G)$ against $\pi:\tilde{G} \to G$.
    Pulling back the half-braiding, we obtain a new centre piece
    $(\bar{a},\pi^\ast \gamma) \in \DZ \tilde{G}_{\pi^\ast k}$.
    Indeed, the hexagon equation $d\gamma = \omega(-,-|a)$ is preserved under pullback.
    The lifts of $a$ form a $Z$-torsor, and the induced quadratic form $q$ on 
    the group of lifts is invariant under the $Z$-action:
    $q(\bar{a} + z)=\pi^\ast\gamma(\bar{a} + z)=\gamma(a)=q(\bar{a})$ for all $z \in Z$.
    This construction extends to a smooth functor $\DZ G_k \to Z^\bot/Z$,
    which preserves the quadratic form.
    We now build the inverse functor.
    
    The pullback cocycle $\pi^\ast \omega$ representing the associator on 
    $\tilde{G}_{\pi^\ast k}$ can be chosen to be equivariant under translation 
    by $Z$ in all variables. 
    The hexagon equation implies that for any centre piece 
    $(\bar{a}, \gamma) \in \DZ \tilde{G}_{\pi^\ast k}$, 
    $z \in Z$, and $\bar{x},\bar{x}' \in \tilde{G}$,
    \begin{equation*}
        \gamma(\bar{x} + z) / \gamma(\bar{x}) = \gamma(\bar{x}' + z) / \gamma(\bar{x}').
    \end{equation*}
    The projection $\pi:\tilde{G} \to G$ defines a simplicial cover 
    $\pi_\bullet: \B \tilde{G}_\bullet \onto \B G_\bullet$, and $\gamma$ is a Segal-Mitchison
    cochain for $G$ defined with respect to this cover. 
    One may check that $\gamma$ satisfies the 
    cocycle condition precisely if it is $Z$-equivariant:
    $\gamma(\bar{x} + z)=\gamma(\bar{x})$ for all $\bar{x} \in \tilde{G}, z \in Z$. 
    By the above equation, it is in fact enough to check it for a single 
    $\bar{x} \in \tilde{G}$.
    The $Z$-equivariance of $\pi^\ast\omega$ implies
    that any $Z$-translate $\bar{a} + z$ of $\bar{a}$ admits $\gamma$ as a half-braiding.
    The quadratic form $q$ sends $(\bar{a},\gamma) \mapsto \gamma(\bar{a})$ 
    and
    $(\bar{a} + z,\gamma) \mapsto \gamma(\bar{a} + z)$. But $(\bar{a}, \gamma) \in Z^\bot$ 
    implies $\gamma(\bar{a} + z)=\gamma(\bar{a})$. Hence, the functor $Z^\bot \to \DZ G_k$ 
    descends to a smooth functor $Z^\bot/Z \to \DZ G_k$, which is manifestly inverse 
    to the construction in the first half of the proof.
\end{proof}

\begin{rmk}
\label{rmk:alternativeProofOfQuotientRule}
The above calculation admits a more abstract description in the setting of additive
monoidal categories. 
Denote by $\cat{C}^\dirSum$ the direct-sum completion 
of a (smooth) monoidal category $\cat{C}$.
Then $G_k^\dirSum$ is a monoidal module category over 
$\tilde{G}_{\pi^\ast k}^\dirSum$, obtained by taking 
modules over the algebra corresponding to 
$Z \subset \tilde{G}$.
The category of modules over an algebra receives a monoidal structure 
precisely when the algebra is commutative,
which happens here if and only if the quadratic form vanishes on $Z$.
The Drinfel'd centre of the category $\cat{C}_A$ of modules over a
commutative algebra $A \in \DZ \cat{C}$ was computed 
in~\cite{schauenburg2001monoidal} 
(under completeness conditions which
are satisfied here): it is the category 
$\DZ \cat{C}_A = \mathrm{loc}_A\DZ \cat{C}$ of \emph{local} $A$-modules
in $\DZ \cat{C}$. 
In the case at hand, 
$\DZ G_k^\dirSum = \mathrm{loc}_Z\DZ \tilde{G}_{\pi^\ast k}^\dirSum=
(Z^\bot/Z,q\restriction_{Z^\bot})^\dirSum$. 
The statement about 2-groups
can be recovered by restricting to simple objects.
\end{rmk}

\begin{example}
\label{ex:DZofSO4}
    Let $G = \SO(4)$. It is the quotient $\Spin(4)/Z$,
    where $Z=\IZ/2\into \IZ/2 \times \IZ/2 = Z(\Spin(4))$ is the
    diagonal copy of $\IZ/2$ in the centre (under the decomposition
    $\Spin(4)=\SU(2) \times \SU(2)$).
    We use Theorem~\ref{thm:nonSimplyConn} to compute the Drinfel'd centre 
    of $G_{\underline k}$, for $\underline k \in \H^4(\B \SO(4),\IZ) \subset
    \H^4(\B \Spin(4),\IZ)$. The isomorphism 
    $\Spin(4)=\SU(2) \times \SU(2)$ allows us to identify the 
    relevant cohomology group as freely generated by the second Chern class 
    in each factor: 
    $\H^4(\B \Spin(4),\IZ) = \H^4(\B \SU(2) \times \B\SU(2),\IZ) 
    = \IZ\langle c_l, c_r \rangle $. 
    We write the cohomology class of the associator in this basis as 
    $\underline k = k_l \cdot c_l + k_r \cdot c_r$.
    By Theorem~\ref{thm:DrinfeldCentreOfGk}, the Drinfel'd centre of $\Spin(4)$ at
    level $\underline k$ is given by 
    \begin{equation*}
    \DZ\Spin(4)_{\underline k}= \left( \IZ/2 \times \IZ/2, q:
    \begin{array}{rl}
            (-1,1)  &\mapsto \exp \tfrac{k_l}{4} \\
            (1,-1)  &\mapsto \exp \tfrac{k_r}{4} \\
            (-1,-1) &\mapsto \exp \tfrac{k_l+k_r}{4}
    \end{array}
    \right).
    \end{equation*}
    The quadratic form $q$ is trivial on $Z$ iff 
    $k_l+k_r \equiv 0 \mod 4$.
    This equation cuts out the subspace of cohomology classes in
    the image of $\H^4(\B \SO(4),\IZ) \into \H^4(\B \Spin(4),\IZ)$.
    The Drinfel'd centre of $G_{k_l,k_r}$ is generated by 
    $[(-1,1)]=[(1,-1)] \in Z(\Spin(4))/Z$ if 
    $(-1,1) \in Z^\bot$, otherwise it is trivial. 
    This is equivalent to the condition that $k_l$ be even.
    In conclusion (see Notation~\ref{not:braidedZ2}),
    \begin{equation*}
        \DZ \SO(4)_{k_l,k_r \in 4\IZ-k_l} = 
        \begin{cases}
            \Vec_{\IZ/2}^\times & k_l \equiv 0 \mod 4 \\
            \sVec^\times & k_l \equiv 2 \mod 4 \\
            \Vec^\times & \text{else}.
        \end{cases}
    \end{equation*}
    The computations in~\cite{vcadek19984} 
    relate the classes $c_l$ and $c_r$ to the more familiar
    first Pontryagin class $p_1$ and Euler class $\chi$, 
    which generate $\H^4(\B \SO(4), \IZ)$.
    This allows the rephrasing of the above result that we 
    gave in the introduction.
\end{example}

\subsection*{Comparison to loop group representations}
We end by comparing $\DZ G_k$ 
(as computed by Theorems \ref{thm:DrinfeldCentreOfGk} and \ref{thm:nonSimplyConn})
to ${(\Rep^k LG)}^\times$, the maximal sub-2-group
of the category of positive energy representations of the loop group.

Let $G$ be a semisimple compact connected Lie group, ie.\ a Lie group 
of the form $G=(\Pi_i G_i) / Z$, where all $G_i$ are compact simple simply-connected 
Lie groups and $Z$ is a finite central subgroup of the product.
Let $k \in \H^4(\B G, \IZ)$ be a cohomology class corresponding 
to a positive-definite bilinear form $I$ (under the identification made
in Section~\ref{sec:stringGrps}).
Below, we show that the braided 2-groups computed in the above subsections 
are equal to ${(\Rep^k LG)}^\times$,
as long as we explicitly exclude any factors of  $E_8$ at level $k=2$.

We use the model of $\Rep^k LG$ as representations
of a unitary vertex operator algebra: $\Rep V_{G,k}$ ---
see~\cite[Eq (2)]{henriques2017classification} for a definition
of $V_{G,k}$ and its modules
and~\cite[Sect 3]{henriques2017chern} for a discussion of alternative 
models of $\Rep^k LG$ and their relation.
More than just tensor categories, these categories are
\emph{unitary modular tensor categories}~\cite{gui2019unitarity}.
When we write ${(\Rep^k LG)}^\times$, we mean the maximal \emph{unitary} 
sub-2-group: we only retain $\tensor$-invertible objects and unitary 
morphisms. This way, we obtain a 2-group with 
\[
\pi_1 {(\Rep^k LG)}^\times=U(1)
\]
(rather than $\IC^\times$).
By Lemma~\ref{lem:pi1ZCispi1C}, the fundamental group
$\pi_1 \DZ G_k = U(1)$ agrees with that of ${(\Rep^k LG)}^\times$.
It remains to compare $\pi_0$ and the induced quadratic form.
We will need the following result.
\begin{lem}
\label{lem:braidingIsTwist}
    Let $X$ be an invertible object in a 
    unitary modular tensor category. Its self-braiding $\beta_{X,X}$ 
    (considered as a complex number via the canonical isomorphism 
    $\End(X \tensor X) = \IC$ sending $\id_{X \tensor X} \mapsto 1$)
    is equal to its \emph{ribbon twist} $\theta_X$.
\end{lem}
\begin{proof}
    This is best understood using string diagrams. The values of $\beta_{X,X}$
    and $\theta_X$ are captured by the following equations (all strands
    are coloured with the object $X$).
    \begin{equation*}
    \begin{tikzpicture}[scale=0.5]
        \draw[-] (2,0) .. controls (2,2) and (0,2) .. (0,4);
        \draw[line width=1em, white] (0,0) .. controls (0,2) and (2,2) .. (2,4);
        \draw[-] (0,0) .. controls (0,2) and (2,2) .. (2,4);
        
        \node at (4,2) {$= \beta_{X,X}\ \cdot$};
        
    \begin{scope}[shift={(5.8,0)}]
        \draw[-] (0,0) -- (0,4);
        \draw[-] (2,0) -- (2,4);
    \end{scope}
    
    \begin{scope}[shift={(15,0)}]
        \draw[-] (0,4) .. controls (1,1) and (2,1) .. (2,2);
        \draw[line width=1em, white] (0,0) .. controls (1,3) and (2,3) .. (2,2);
        \draw[-] (0,0) .. controls (1,3) and (2,3) .. (2,2);
        \node at (3.3,2) {$= \theta_X\ \cdot$};
    \begin{scope}[shift={(4.8,0)}]
        \draw[-] (0,0) -- (0,4);
    \end{scope}
    \end{scope}
    \end{tikzpicture}
    \end{equation*}
    The left hand sides of these equations become equal upon taking the trace ---
    they both form the figure eight.
    The traces of the right hand sides must thus also be equal:
    \begin{equation*}
        \beta_{X,X} \cdot \mathrm{Tr}(\id_{X \tensor X}) = 
        \theta_X \cdot \mathrm{Tr}(\id_X).
    \end{equation*}
    In a unitary tensor category, $\mathrm{Tr}(\id_Y)=1$ if $Y$ is invertible.
    Invertibility of $X$ and $X \tensor X$ now implies $\beta_{X,X} = \theta_X$.
\end{proof}

Armed with this Lemma, we proceed as before: we first compute the simply-connected
case, and then deal with quotients.
\begin{thm}
\label{thm:simplyConnAgreesWithLG}
    For a compact simply-connected Lie group $G$ and positive-definite level
    $k \in \H^4(\B G, \IZ)$, there is a braided equivalence
    \begin{equation*}
        \DZ G_k \simeq {(\Rep V_{G,k})}^\times,
    \end{equation*}
    unless $G$ contains a factor of $E_8$ to which
    $k$ restricts as level 2.
\end{thm}
\begin{proof}
    The group of invertible objects of $\Rep V_{G,k}$ was identified as
    \begin{equation*}
        \pi_0 {(\Rep V_{G,k})}^\times = Z(G)
    \end{equation*}
    in~\cite[Prop 2.20]{li2001certain} 
    (see also~\cite[Prop 7]{henriques2017classification}).
    This only fails when $G$ contains a factor of $E_8$ 
    and $k$ restricts to level $2$ on that factor.
    The category $\Rep V_{E_8,2}$ contains an 
    invertible object of order 2, despite $Z(E_8)$ being trivial. 
    Excluding this case, 
    ${(\Rep V_{G,k})}^\times$ is tensor equivalent to the 
    underlying 2-group of $\DZ G_k $, and it remains to compare the self-braidings.
    
    By Lemma~\ref{lem:braidingIsTwist}, the self-braiding on an invertible object is 
    equal to its ribbon twist. The values of these are well known:
    On an invertible object corresponding to $z \in Z(G)$, 
    \begin{equation*}
        \beta_{z,z} = \theta_z = \exp\left(\tfrac{1}{2}I(\bar{z},\bar{z})\right),
    \end{equation*}
    where $\bar{z}$ again denotes a lift of $z$ to a Cartan of $G$
    (see~\cite[Prop 10]{henriques2017classification}). 
    This formula is the same as that computed for simply-connected groups
    in Theorem~\ref{thm:DrinfeldCentreOfGk} and Lemma~\ref{lem:DZofProduct}:
    the quadratic form on $\pi_0 {(\Rep^k LG)}^\times$ agrees with that
    on $\pi_0 \DZ G_k$.
\end{proof}

To show the above result also holds for non-simply-connected semisimple 
compact Lie groups $G$, 
we use the relation between ${(\Rep V_{G,k})}^\times$ and 
${(\Rep V_{\tilde{G},\tilde{k}})}^\times$, where $\tilde{G}$ denotes
the universal cover of $G$ and $\tilde{k}$ the pullback 
cohomology class. Recall that $G$ is of the form $\tilde{G}/Z$ where
$Z \into \tilde{G}$ is a finite central subgroup.

\begin{thm}
\label{thm:semisimpleAgreesWithLG}
    Let $G=\tilde{G}/Z$ be a semisimple compact connected Lie group
    and $k \in \H^4(\B G, \IZ)$ positive-definite
    (such that $\tilde{G}$ contains no factor of $E_8$ to 
    which $k$ pulls back as level 2).
    Then there is a braided equivalence
    \begin{equation*}
        \DZ G_k = {(\Rep V_{G,k})}^\times.
    \end{equation*}
\end{thm}
\begin{proof}
    By~\cite[Thm 3.4]{huang2015braided} and~\cite{creutzig2017tensor}, 
    there is a braided equivalence
    \begin{equation*}
        \Rep V_{G,k} = \mathrm{loc}_Z \Rep V_{\tilde{G},\tilde{k}},
    \end{equation*}
    where the right-hand side is 
    the category of local $Z$-modules in $\Rep V_{\tilde{G},\tilde{k}}$.
    The maximal sub-2-group of the right-hand side may be computed by 
    first restricting to the additive subcategory on invertible objects,
    taking local modules there, and then taking the maximal 
    sub-2-group. 
    (See also Remark~\ref{rmk:alternativeProofOfQuotientRule}
    for a computation of $\DZ G_k$ using this method.)
    The additive subcategory on invertibles is a pointed braided fusion category
    with simple objects $Z(\tilde{G})$ and braiding described by a 
    quadratic form $q$ (which agrees with that on $\DZ \tilde{G}_{\tilde{k}}$ 
    by Theorem~\ref{thm:simplyConnAgreesWithLG}). 
    The category of local $Z$-modules is described 
    in~\cite[Sect 3.1]{davydov2018third}: the local modules form a pointed braided
    fusion category with group of simple objects $Z^\bot/Z$ 
    (where $Z^\bot$ is defined as before Theorem~\ref{thm:nonSimplyConn}).
    The quadratic form on $Z^\bot/Z$ is given by the restriction 
    of $q$ to $Z^\bot$. 
    Thus the maximal sub-2-group is
    \begin{equation*}
        {\left(\mathrm{loc}_Z \Rep V_{\tilde{G},\tilde{k}}\right)}^\times =
        \left(Z^\bot/Z, q\restriction_{Z^\bot}\right).
    \end{equation*}
    This agrees with the result of Theorem~\ref{thm:nonSimplyConn},
    completing the comparison.
\end{proof}


\printbibliography

\end{document}